\documentclass[11pt]{amsart}

\usepackage{hyperref}
\hypersetup{bookmarksdepth=3}
  
\usepackage{geometry}
\usepackage{amsmath,amssymb}
\usepackage{
mathrsfs}
\usepackage{enumerate}
\usepackage{esint}

\usepackage{tensor}

\usepackage{amssymb,amsmath,amsthm}

\newtheorem{theorem}{Theorem}
\newtheorem{lemma}[theorem]{Lemma}

\newtheorem{definition}[theorem]{Definition}

\theoremstyle{remark}\newtheorem{remark}[theorem]{Remark}

\newtheorem{proposition}[theorem]{Proposition}

\usepackage{graphicx}

\newcommand{\R}{{\mathbb R}}

\newcommand{\Z}{{\mathbb Z}}

\newcommand{\C}{{\mathbb C}}

\begin{document}

{\let\thefootnote\relax\footnote{Date: 20th April 2018. 

\textcopyright 2018 by the authors. Faithful reproduction of this article, in its entirety, by any means is permitted for noncommercial purposes.}}

\title{Nonlinear Schr\"odinger equation, differentiation by parts and modulation spaces. }

\author{L. Chaichenets}
\address{leonid chaichenets, department of mathematics, institute for analysis, karlsruhe institute of technology, 76128 karlsruhe, germany }
\email{leonid.chaichenets@kit.edu}

\author{D. Hundertmark}
\address{dirk hundertmark, department of mathematics, institute for analysis, karlsruhe institute of technology, 76128 karlsruhe, germany }
\email{dirk.hundertmark@kit.edu}

\author{P. Kunstmann}
\address{peer kunstmann, department of mathematics, institute for analysis, karlsruhe institute of technology, 76128 karlsruhe, germany }
\email{peer.kunstmann@kit.edu}

\author{N. Pattakos}
\address{nikolaos pattakos, department of mathematics, institute for analysis, karlsruhe institute of technology, 76128 karlsruhe, germany }
\email{nikolaos.pattakos@kit.edu}

\begin{abstract}
{We show the existence of weak solutions in the extended sense of the Cauchy problem for the cubic nonlinear Schr\"odinger equation in the modulation space $M_{p,q}^{s}(\mathbb R)$ where $1\leq q\leq2$, $2\leq p<\frac{10q'}{q'+6}$ and $s\geq0$. Moreover, for either $1\leq q\leq\frac32, s\geq0$ and $2\leq p\leq 3$  or $\frac32<q\leq\frac{18}{11}, s>\frac23-\frac1{q}$ and $2\leq p\leq 3$ or $\frac{18}{11}<q\leq2, s>\frac23-\frac1{q}$ and $2\leq p<\frac{10q'}{q'+6}$ we show that the Cauchy problem is unconditionally wellposed in $M_{p,q}^{s}(\R).$ This improves \cite{NP}, where the case $p=2$ was considered and the differentiation by parts technique was introduced to a problem with continuous Fourier variable. Here the same technique is used, but more delicate estimates are necessary for $p\neq2$.}
\end{abstract}

\maketitle
\pagestyle {myheadings}

\begin{section}{introduction and main results}
\markboth{\normalsize L. Chaichenets, D. Hundertmark, P. Kunstmann and N. Pattakos }{\normalsize  NLS, Differentiation by Parts and $M_{p,q}^{s}(\mathbb R)$}
We are interested in the cubic nonlinear Schr\"odinger equation defined by 

\begin{equation}
\label{maineq}
\begin{cases} iu_{t}-u_{xx}\pm|u|^{2}u=0 &,\ (t,x)\in\mathbb R^{2}\\
u(0,x)=u_{0}(x) &,\ x\in\mathbb R\\
\end{cases}
\end{equation}
with initial data $u_{0}$ in the modulation space $M_{p,q}^{s}(\mathbb R).$ The precise definition of these spaces is in Section \ref{prell}. Our goal is to study the existence and unconditional uniqueness for the PDE (\ref{maineq}). From \cite{FEI} (Proposition $6.9$) it is known that for $s>1/q'$ or $s\geq 0$ and $q=1$ the modulation space $M^{s}_{p,q}(\mathbb R)$ is a Banach algebra and therefore an easy Banach contraction principle argument implies that NLS (\ref{maineq}) is locally wellposed for $u_{0}\in M^{s}_{p,q}(\mathbb R)$ with solution $u\in C([0,T];M^{s}_{p,q}(\mathbb R))$, $T>0$ (see \cite{BO}). In this paper, with a different approach, we are able to cover the remaining cases $0\leq s\leq1/q'$, unfortunately not for all values of $p.$ The differentiation by parts technique that we apply was also used in \cite{BIT} to attack similar problems for the KdV equation but with periodic initial data. In \cite{GKO} this technique was used to prove unconditional wellposedness of the periodic cubic NLS in one dimension and in \cite{HY} to attack the same problem for the cubic NLS and the mKdV on the real line with initial data in the Sobolev spaces $H^{s}(\R).$ Since $H^{s}(\R)=M_{2,2}^{s}(\R)$ (see Section \ref{prell} below) modulation spaces can be viewed as a certain refinement of the scale of $L^{2}$-based Sobolev spaces. Our approach is different from \cite{HY} and it seems to be quite natural for modulation spaces. In this paper our initial data is far from being periodic, and for this reason there are some major differences and some difficulties that do not occur in the periodic setting, which were pointed out in \cite{NP} too, where the case $p=2$ was considered. The main difference between this paper and \cite{NP} is that we are able to obtain estimates on the $L^{p}$ norm of the operators $R^{J,t}_{T^0,\mathbf n}$ (see (\ref{oops1})) for $p\neq2$ through an $L^\infty$ estimate and an interpolation argument. 

In order to give a meaning to solutions of the NLS in $C([0,T], M_{p,q}^{s}(\R))$ and to the nonlinearity $\mathcal N(u):=u|u|^{2}$ we need the following definitions which first appeared in \cite{C1}.

\begin{definition}
\label{def1}
For fixed $1\leq p\leq\infty$, a sequence of Fourier cutoff operators is a sequence of Fourier multiplier operators $\{T_{N}\}_{N\in \mathbb N}$ on $\mathcal S'(\mathbb R)$ with symbols $m_{N}$ such that
\begin{itemize}
\item $m_{N}$ is compactly supported for all $N\in\mathbb N$,
\item $\sup_{N\in\mathbb N}\|T_{N}\|_{L^{p}\to L^{p}}<\infty$ and
\item for every $f$ in a dense subset of $L^{p}(\R)$ we have $\lim_{N\to\infty}\|T_{N}f-f\|_{p}=0$.
\end{itemize}
\end{definition}
Notice that in our definition a sequence of Fourier cutoff operators depends on the given value of $p\in[1,\infty]$ in $M^{s}_{p,q}(\R)$. 

\begin{definition}
\label{def2}
Let $u\in C([0,T],M_{p,q}^{s}(\R)).$ We say that $\mathcal N(u)$ exists and is equal to a distribution $w\in\mathcal S'((0,T)\times\R)$ if for every sequence $\{T_{N}\}_{N\in\mathbb N}$ of Fourier cutoff operators we have
\begin{equation}
\label{wknows}
\lim_{N\to\infty}\mathcal N(T_{N}u)=w,
\end{equation}
in the sense of distributions on $(0,T)\times\R$.
\end{definition}

\begin{definition}
\label{def3}
We say that $u\in C([0,T],M_{p,q}^{s}(\R))$ is a weak solution in the extended sense of NLS (\ref{maineq}) if 
\begin{itemize}
\item $u(0,x)=u_{0}(x)$,
\item the nonlinearity $\mathcal N(u)$ exists in the sense of Definition \ref{def2},
\item $u$ satisfies (\ref{maineq}) in the sense of distributions on $(0,T)\times\R$, where the nonlinearity $\mathcal N(u)=u|u|^{2}$ is interpreted as above.
\end{itemize}
\end{definition}
Our main result which establishes the existence of weak solutions in the extended sense generalises the one in \cite{NP} and it is the following

\begin{theorem}
\label{th1}
Let $s\geq0$, $1\leq q\leq2$ and $2\leq p<\frac{10q'}{q'+6}$. For $u_{0}\in M_{p,q}^{s}(\mathbb R)$ there exists a weak solution in the extended sense $u\in C([0,T];M_{p,q}^{s}(\mathbb R))$ of NLS (\ref{maineq}) with initial condition $u_{0}$ in the sense of Definition \ref{def3}, where the time $T$ of existence depends only on $\|u_{0}\|_{M_{p,q}^{s}}$. Moreover, the solution map is Lipschitz continuous. 
\end{theorem}

\begin{remark}
The restriction on the range of $p$ is dictated by the construction of our solution of the NLS. More precisely, we decompose the NLS into countably many "smaller" parts and at the end we sum all of them together. In order for this summation to make sense all the series must by convergent in the appropriate spaces and as a consequence we obtain the restriction $p<\frac{10q'}{q'+6}$ (see the remarks after (\ref{factori}) below). The restriction on $q$ comes from the estimate of the resonant operator $R_{2}^{t}$ in Lemma \ref{lem}.  
\end{remark}
The next theorem is about the unconditional wellposedness of NLS (\ref{maineq}) with initial data in a modulation space, that is, uniqueness in $C([0,T],M_{p,q}^{s}(\R))$ without intersecting with any auxiliary function space (see \cite{TK} where this notion first appeared).

\begin{theorem}
\label{mainyeah}
For $u_{0}\in M_{p,q}^{s}(\R)$ with either $s\geq0, 2\leq p\leq 3$ and $1\leq q\leq\frac32$ or $s>\frac23-\frac1{q}, 2\leq p\leq 3$ and $\frac32<q\leq\frac{18}{11}$ or $s>\frac23-\frac1{q}, 2\leq p<\frac{10q'}{q'+6}$ and $\frac{18}{11}<q\leq2$ the solution $u$ with initial condition $u_{0}$ constructed in Theorem \ref{th1} is unique in $C([0,T],M_{p,q}^{s}(\R))$. 
\end{theorem}

\begin{remark}
In \cite{HY} it is proved that (\ref{maineq}) is unconditionally locally wellposed in $H^{s}(\R)$ for $s\geq\frac16.$ For $s>\frac16$ these spaces are covered by Theorem \ref{mainyeah}. Later in the proof it is easy to notice that the only requirement is that the initial data lies in a space that embeds in $L^{3}(\R)$ and therefore, our calculations are still applicable for the remaining space $H^{\frac16}(\R)$. 
\end{remark}

The paper is organised as follows: In Section \ref{prell} we define modulation spaces and we present some preliminary results that are going to be used throughout the proofs of the main theorems. In Section \ref{firstep} the first steps of the iteration process are presented and in Section \ref{treeind} the tree notation and the induction step finish the infinite iteration procedure. Then, in Section \ref{thth4} Theorem \ref{th1} is proved where the solution is constructed through an approximation by smooth solutions and in Section \ref{thth6} the unconditional uniqueness of Theorem \ref{mainyeah} is presented under the extra assumption that the solution lies in the space $C([0,T],L^{3}(\R)).$
\end{section}

\begin{section}{Preliminaries} 
\label{prell}

To state the definition of a modulation space we need to fix some notation. We will denote by $S'(\R)$ the space of tempered distributions and by $D'(\R)$ the space of distributions. Let $Q_{0}=[-\frac12, \frac12)$ and its translations $Q_{k}=Q_{0}+k$ for all $k\in\mathbb Z.$ Consider a partition of unity $\{\sigma_{k}=\sigma_{0}(\cdot-k)\}_{k\in\mathbb Z}\subset C^{\infty}(\mathbb R)$ satisfying 

\begin{itemize}
\item
$ \exists c > 0: \,
\forall \eta \in Q_{0}: \,
|\sigma_{0}(\eta)| \geq c$,
\item
$
\mbox{supp}(\sigma_{0}) \subseteq \{\xi\in\R:|\xi|<1\}=:B(0,1)$,
\end{itemize}
and define the isometric decomposition operators
\begin{equation}
\label{iso}
\Box_{k} := \mathcal F^{(-1)} \sigma_{k} \mathcal F, \quad
\left(\forall k \in \Z \right).
\end{equation}
For $\Lambda:=\{-1,0,1\}$ notice that
\begin{equation}
\label{Lamlam}
k\notin \Lambda \Rightarrow \Box_{n}\Box_{n+k}=0.
\end{equation}
Then the norm of a tempered distribution $f\in S'(\R)$ in the modulation space $M^{s}_{p,q}(\mathbb R)$, where $s\in\mathbb R, 1\leq p,q\leq\infty$, is 
\begin{equation}
\label{def}
\|f\|_{M^{s}_{p,q}}:=\Big(\sum_{k\in\mathbb Z}\langle k\rangle^{sq}\|\Box_{k}f\|_{p}^{q}\Big)^{\frac1{q}},
\end{equation}
with the usual interpretation when the index $q$ is equal to infinity, where we denote by $\langle k\rangle=(1+|k|^{2})^{\frac{1}{2}}$ the Japanese bracket. It can be proved that different choices of the function $\sigma_{0}$ lead to equivalent norms in $M^{s}_{p,q}(\mathbb R).$ Later, during the proof of the main theorem we will make use of this fact. When $s=0$ we denote the space $M^{0}_{p,q}(\mathbb R)$ by $M_{p,q}(\mathbb R).$ In the special case where $p=q=2$ we have $M_{2,2}^{s}(\R)=H^{s}(\R)$ the usual Sobolev spaces. In our calculations we are going to use that for $s>1/q'$ and $1\leq p, q\leq\infty$, the embedding 
\begin{equation}
\label{yeye}
M_{p,q}^{s}(\R)\hookrightarrow C_{b}(\R)=\{f:\R\to\C\ |\ f\ \mbox{continuous and bounded}\},
\end{equation}
and for $\Big(1\leq p_{1}\leq p_{2}\leq \infty$, $1\leq q_{1}\leq q_{2}\leq\infty$, $s_{1}\geq s_{2}\Big)$ or $\Big(1\leq p_{1}\leq p_{2}\leq \infty$, $1\leq q_{2}<q_{1}\leq\infty$, $s_{1}>s_{2}+\frac1{q_{2}}-\frac1{q_{1}}\Big)$ the embedding
\begin{equation}
\label{yeye233}
M_{p_{1}, q_{1}}^{s_{1}}(\R)\hookrightarrow M_{p_{2}, q_{2}}^{s_{2}}(\R),
\end{equation}
are both continuous and can be found in \cite{FEI} (Proposition $6.8$ and Proposition $6.5$). In that paper modulation spaces were introduced for the first time by Feichtinger and since then they have been used extensively in the study of nonlinear dispersive equations. They have become canonical for both time-frequency and phase-space analysis. See \cite{BH} for many of their properties such as embeddings in other known function spaces and equivalent expressions for their norm. Also, by \cite{BH} it is known that for any $1<p\leq\infty$ we have the embedding $M_{p,1}(\R)\hookrightarrow L^{p}(\R)\cap L^{\infty}(\R)$ which together with the fact that $M_{2,2}(\R)=L^{2}(\R)$ and complex interpolation, imply that for any $p\in[2,\infty]$ we have the embedding $M_{p,p'}(\R)\hookrightarrow L^{p}(\R)$. Later in Section \ref{thth6} we will use this fact for $p=3$, that is 
\begin{equation}
\label{hhh}
M_{3,\frac32}(\R)\hookrightarrow L^{3}(\R).
\end{equation}
To conclude this section we need that for $S(t)=e^{it\Delta}$ the Schr\"odinger semigroup we have the estimate:
\begin{equation}
\label{Sch}
\|S(t)f\|_{M_{p,q}^{s}}\lesssim(1+|t|)^{|\frac12-\frac1{p}|}\|f\|_{M_{p,q}^{s}},
\end{equation}
where the implicit constant does not depend on $f, t.$ We also need the following corollary of Young's inequality (see \cite{BH}, Proposition $1.9$).

\begin{lemma}
\label{Bern}
Let $1\leq p\leq \infty$ and $\sigma \in C^{\infty}_{c}(\R)$. Then the multiplier operator $T_\sigma: S'(\R) \to S'(\R)$ defined by
\begin{equation*}
(T_\sigma f) = \mathcal F^{-1}(\sigma \cdot \hat{f}), \quad
\forall f \in S'(\R)
\end{equation*}
is bounded on $L^p(\R)$ and
\begin{equation*}
\|T_{\sigma}\|_{L^p(\R)\to L^p(\R)} \lesssim\|\check{\sigma}\|_{L^{1}(\R)}.
\end{equation*}
\end{lemma}
Another useful corollary is that for $1\leq p_{1}\leq p_{2}\leq\infty$ the following holds
\begin{equation}
\label{Bern1}
\|\Box_{k}f\|_{p_{2}}\lesssim\|\Box_{k}f\|_{p_{1}},
\end{equation}
where the implicit constant is independent of $k$ and the function $f$. 

Lastly, let us recall the following number theoretic fact (see \cite{HW}, Theorem $315$) which is going to be used throughout the proof of Theorem \ref{th1}: Given an integer $m$, let $d(m)$ denote the number of divisors of $m$. Then we have
\begin{equation}
\label{num}
d(m)\lesssim e^{c\frac{\log m}{\log\log m}}=o(m^{\epsilon}),
\end{equation}
for all $\epsilon>0$. 

\end{section}

\begin{section}{first steps of the iteration procedure}
\label{firstep}
The calculations are similar to those presented in \cite{NP} where the difference is that instead of using $L^{2}$ estimates for the Fourier-space variable we use $L^{p}$ estimates which is something that will become clearer in the calculations that follow. Nevertheless, there are a lot of new details that need to be taken care of. For this reason, and for the reader's convenience we will be as detailed as possible.

From here on, we consider only the case $s=0$ in Theorem \ref{th1} since for $s>0$ similar considerations apply. See Remark \ref{reme} at the end of the section for a more detailed argument. 

Also, since our indices $1\leq q<3$ and $2\leq p<\frac{10q'}{q'+6}$ are fixed, we can find a fixed number $A>1$ such that 
\begin{equation}
\label{mpap}
2\leq p<\frac{2q'(2A+3)}{(2A-1)q'+6}.
\end{equation} 
Notice that the function $f(A)=\frac{2q'(2A+3)}{(2A-1)q'+6}$ is decreasing and in the range $A>1$ it has a global maximum at $A=1$. From here on, we choose our bump function $\sigma_{0}$ to satisfy the following bounds on its derivatives
\begin{equation}
\label{bbbo}
\Big\|\frac{d^{J}}{dx^{J}}\ \sigma_{0}\Big\|_{\infty}\lesssim (J!)^{A},
\end{equation}
for all $J\in\mathbb N$. This is crucial for Lemma \ref{indu}. Notice that $A\leq1$ can not be true since then our compactly supported function $\sigma_{0}$ would be a real analytic function and therefore, it would be identically zero.

For $n\in\mathbb Z$ let us define
\begin{equation}
\label{ww0}
u_{n}(t,x)=\Box_{n} u(t,x),
\end{equation}
\begin{equation}
\label{ww1}
v(t,x)=e^{it\partial_{x}^{2}}u(t,x),
\end{equation}
\begin{equation}
\label{ww}
v_{n}(t,x)=e^{it\partial_{x}^{2}}u_{n}(t,x)=\Box_{n}[e^{it\partial_{x}^{2}}u(t,x)]=\Box_{n}v(t,x).
\end{equation}
Also for $(\xi,\xi_{1},\xi_{2},\xi_{3})\in\mathbb R^{4}$ we define the function
$$\Phi(\xi,\xi_{1},\xi_{2},\xi_{3})=\xi^{2}-\xi_{1}^{2}+\xi_{2}^{2}-\xi_{3}^{2},$$
which is equal to 
\begin{equation}
\label{malmal}
\Phi(\xi,\xi_{1},\xi_{2},\xi_{3})=2(\xi-\xi_{1})(\xi-\xi_{3}),
\end{equation}
if $\xi=\xi_{1}-\xi_{2}+\xi_{3}.$ Our main equation (\ref{maineq}) implies that
\begin{equation}
\label{main2}
i\partial_{t}u_{n}-(u_{n})_{xx}\pm\Box_{n}(|u|^{2}u)=0,
\end{equation}
and by calculating ($u=\sum_{k}\Box_{k}u$ in $S'(\R)$)
$$\Box_{n}(u\bar{u}u)=\Box_{n}\sum_{n_{1},n_{2},n_{3}}u_{n_{1}}\bar{u}_{n_{2}}u_{n_{3}}=\sum_{n_{1}-n_{2}+n_{3}\approx n}\Box_{n}[u_{n_{1}}\bar{u}_{n_{2}}u_{n_{3}}],$$
where by $\approx n$ we mean $=n$, or $=n+1$, or $=n-1.$ Next we do the change of variables $u_{n}(t,x)=e^{-it\partial_{x}^{2}}v_{n}(t,x)$ and arrive at the expression
\begin{equation}
\label{main3}
\partial_{t}v_{n}=\pm i\sum_{n_{1}-n_{2}+n_{3}\approx n}\Box_{n}\Big(e^{it\partial_{x}^{2}}[e^{-it\partial_{x}^{2}}v_{n_{1}}\cdot e^{it\partial_{x}^{2}}\bar{v}_{n_{2}}\cdot e^{-it\partial_{x}^{2}}v_{n_{3}}]\Big).
\end{equation}
We define the $1$st generation operators by
\begin{equation}
\label{main4}
Q^{1,t}_{n}(v_{n_{1}},\bar{v}_{n_{2}},v_{n_{3}})(x)=\Box_{n}\Big(e^{it\partial_{x}^{2}}[e^{-it\partial_{x}^{2}}v_{n_{1}}\cdot e^{it\partial_{x}^{2}}\bar{v}_{n_{2}}\cdot e^{-it\partial_{x}^{2}}v_{n_{3}}]\Big),
\end{equation}
and continue with the splitting 
\begin{equation}
\label{main11}
\partial_{t}v_{n}=\pm i\sum_{n_{1}-n_{2}+n_{3}\approx n}Q^{1,t}_{n}(v_{n_{1}},\bar{v}_{n_{2}},v_{n_{3}})=\sum_{\substack{n_{1}\approx n\\ or\\ n_{3}\approx n}}\ldots+\sum_{n_{1}, n_{3}\not\approx n}\ldots,
\end{equation}
where we define the resonant part
\begin{equation}
\label{main9}
R^{t}_{2}(v)(n)-R^{t}_{1}(v)(n)=\Big(\sum_{n_{1}\approx n}Q^{1,t}_{n}+\sum_{n_{3}\approx n}Q^{1,t}_{n}\Big)-\sum_{\substack{n_{1}\approx n\\ and\\ n_{3}\approx n}}Q^{1,t}_{n}(v_{n_{1}},\bar{v}_{n_{2}},v_{n_{3}}),
\end{equation}
and the non-resonant part
\begin{equation}
\label{main10}
N_{1}^{t}(v)(n)=\sum_{n_{1}, n_{3}\not\approx n}Q^{1,t}_{n}(v_{n_{1}},\bar{v}_{n_{2}},v_{n_{3}}),
\end{equation}
which implies the following expression for our NLS (we drop the factor $\pm i$ in front of the sum since they will play no role in our analysis)
\begin{equation}
\label{mainmain}
\partial_{t}v_{n}=R^{t}_{2}(v)(n)-R^{t}_{1}(v)(n)+N_{1}^{t}(v)(n).
\end{equation}
For the resonant part we have the following

\begin{lemma}
\label{lem}
For $j=1,2$ 
$$\|R^{t}_{j}(v)\|_{l^{q}M_{p,q}}\lesssim(1+|t|)^{4|\frac12-\frac1{p}|}\|v\|^{3}_{M_{p,q}},$$
and
$$\|R^{t}_{j}(v)-R^{t}_{j}(w)\|_{l^{q}M_{p,q}}\lesssim(1+|t|)^{4|\frac12-\frac1{p}|}(\|v\|^{2}_{M_{p,q}}+\|w\|^{2}_{M_{p,q}})\|v-w\|_{M_{p,q}}.$$
\end{lemma}

\begin{remark}
\label{normnormm}
In the previous and all the following lemmata we use the $l^{q}M_{p,q}$ norm instead of the $l^{q}L^{p}$ norm to estimate our operators. Since the operators are nicely localised these norms are equivalent. We prefer the former one because the Schr\"odinger operator $e^{it\partial_{x}^{2}}$ is bounded as Lemma \ref{Sch} dictates. 
\end{remark}
\begin{proof}
Let us consider $R^{t}_{1}.$ By its definition, for fixed $n$, $R^{t}_{1}(n)$ consists of finitely many summands, since $|n-n_{1}|, |n-n_{3}|\leq 1$ and $|n-n_{2}|\leq 3$. We will handle $Q^{1,t}_{n}(v_{n},\bar{v}_{n},v_{n})$ since the remaining summands can be treated similarly. Since,
$$Q^{1,t}_{n}(v_{n},\bar{v}_{n},v_{n})=\Box_{n}\Big(e^{it\partial_{x}^{2}}[e^{-it\partial_{x}^{2}}v_{n}\cdot e^{it\partial_{x}^{2}}\bar{v}_{n}\cdot e^{-it\partial_{x}^{2}}v_{n}]\Big)$$
its $M_{p,q}$ norm is bounded from above by
$$\|e^{it\partial_{x}^{2}}\Box_{n}(|u_{n}|^{2}u_{n})\|_{M_{p,q}}\lesssim(1+|t|)^{|\frac12-\frac1{p}|}\|\Box_{n}(|u_{n}|^{2}u_{n})\|_{M_{p,q}},$$
where we used (\ref{Sch}). By estimating this last norm we have
$$\|\Box_{n}(|u_{n}|^{2}u_{n})\|_{M_{p,q}}=\Big(\sum_{m\in\Z}\|\Box_{m}\Box_{n}(|u_{n}|^{2}u_{n})\|_{p}^{q}\Big)^{\frac1{q}}\lesssim$$
$$\Big(\sum_{l\in\Lambda}\|\Box_{n+l}(|u_{n}|^{2}u_{n})\|_{p}^{q}\Big)^{\frac1{q}}\lesssim\||u_{n}|^{2}u_{n}\|_{p}=\|u_{n}\|^{3}_{3p},$$
where in both inequalities we used Lemma (\ref{Bern}) and implication \ref{Lamlam}. With the use of (\ref{Bern1}) we have $\|u_{n}\|_{3p}\lesssim\|u_{n}\|_{p}$ and by taking the $l^{q}$ norm in the discrete variable we arrive at the upper bound
$$(1+|t|)^{|\frac12-\frac1{p}|}\Big(\sum_{n\in\Z}\|u_{n}\|_{p}^{3q}\Big)^{\frac1{q}}\leq(1+|t|)^{|\frac12-\frac1{p}|}\|u\|_{M_{p,q}}^{3},$$
where we used the embedding $l^{q}\hookrightarrow l^{3q}$. Since $u=e^{-it\partial_{x}^{2}}v$ another application of (\ref{Sch}) gives us the desired upper bound. 

For the $R_{2}^{t}$ operator, it suffices to estimate the sum
$$\sum_{\substack{n_{1}-n_{2}+n_{3}\approx n\\ n_{1}\approx n}}Q^{t}_{n}(v_{n_{1}},\bar{v}_{n_{2}},v_{n_{3}})$$
which consists of finitely many sums depending on whether $n_{1}=n-1$, or $n_{1}=n$, or $n_{1}=n+1.$ Let us only treat 
$$\Box_{n}\ e^{it\partial_{x}^{2}}\ \Big(e^{-it\partial_{x}^{2}}v_{n}\sum_{n_{2}\in\Z}|e^{-it\partial_{x}^{2}}v_{n_{2}}|^{2}\Big),$$
since for the remaining sums similar considerations apply. By (\ref{Sch}) its $M_{p,q}$ norm is bounded from above by 
$$(1+|t|)^{|\frac12-\frac1{p}|}\Big\|\Box_{n}u_{n}\sum_{n_{2}\in\Z}|u_{n_{2}}|^{2}\Big\|_{M_{p,q}}=(1+|t|)^{|\frac12-\frac1{p}|}\Big(\sum_{m\in\Z}\Big\|\Box_{m}\Box_{n}u_{n}\sum_{n_{2}\in\Z}|u_{n_{2}}|^{2}\Big\|_{p}^{q}\Big)^{\frac1{q}},$$
and the term in parenthesis is equal to 
$$\Big(\sum_{l\in\Lambda}\Big\|\Box_{n+l}\Box_{n}u_{n}\sum_{n_{2}\in\Z}|u_{n_{2}}|^{2}\Big\|_{p}^{q}\Big)^{\frac1{q}},$$
where we used implication (\ref{Lamlam}). Again from (\ref{Bern}), the fact that $\Lambda$ is a finite set, H\"older's inequality and (\ref{Bern1}) the sum can be controlled by 
$$\Big(\sum_{l\in\Lambda}\Big\|u_{n}\sum_{n_{2}\in\Z}|u_{n_{2}}|^{2}\Big\|_{p}^{q}\Big)^{\frac1{q}}\lesssim\sum_{n_{2}\in\Z}\|u_{n}|u_{n_{2}}|^{2}\|_{p}\leq\sum_{n_{2}\in\Z}\|u_{n}\|_{2p}\|u_{n_{2}}\|_{4p}^{2}\lesssim\|u_{n}\|_{p}\sum_{n_{2}\in\Z}\|u_{n_{2}}\|_{p}^{2}.$$
This last term is equal to $\|u_{n}\|_{p}\|u\|_{M_{p,2}}^{2}$ and since $1\leq q\leq2$ we know that the embedding $l^{q}\hookrightarrow l^{2}$ implies the upper bound $\|u_{n}\|_{p}\|u\|_{M_{p,q}}^{2}.$ Finally, by taking the $l^{q}$ in the discrete variable we obtain that 
$$\|R_{2}^{t}(v)\|_{l^{q}M_{p,q}}\lesssim(1+|t|)^{4|\frac12-\frac1{p}|}\|v\|_{M_{p,q}}^{3}.$$
For the difference part $R_{1}^{t}(v)-R_{1}^{t}(w)$ we have to estimate terms of the following form $\Box_{n}e^{it\partial_{x}^{2}}(e^{-it\partial_{x}^{2}}v_{n})^{2}(e^{-it\partial_{x}^{2}}v_{n}-e^{-it\partial_{x}^{2}}w_{n})$ in the $l^{q}M_{p,q}$ norm. As before, from (\ref{Sch}) the $M_{p,q}$ norm is bounded above by 
$$(1+|t|)^{|\frac12-\frac1{p}|}\|\Box_{n}(e^{-it\partial_{x}^{2}}v_{n})^{2}(e^{-it\partial_{x}^{2}}v_{n}-e^{-it\partial_{x}^{2}}w_{n})\|_{M_{p,q}},$$
and this last norm is equal to
$$\Big(\sum_{m\in\Z}\|\Box_{m}\Box_{n}(e^{-it\partial_{x}^{2}}v_{n})^{2}(e^{-it\partial_{x}^{2}}v_{n}-e^{-it\partial_{x}^{2}}w_{n})\|_{p}^{q}\Big)^{\frac1{q}}=$$
$$\Big(\sum_{l\in\Lambda}\|\Box_{n+l}\Box_{n}(e^{-it\partial_{x}^{2}}v_{n})^{2}(e^{-it\partial_{x}^{2}}v_{n}-e^{-it\partial_{x}^{2}}w_{n})\|_{p}^{q}\Big)^{\frac1{q}}\lesssim\|(e^{-it\partial_{x}^{2}}v_{n})^{2}(e^{-it\partial_{x}^{2}}v_{n}-e^{-it\partial_{x}^{2}}w_{n})\|_{p},$$
where we used (\ref{Bern}). Applying H\"older's inequality and (\ref{Bern1}) we arrive at
$$\|e^{-it\partial_{x}^{2}}v_{n}\|_{4p}^{2}\|e^{-it\partial_{x}^{2}}v_{n}-e^{-it\partial_{x}^{2}}w_{n}\|_{2p}\lesssim\|e^{-it\partial_{x}^{2}}v_{n}\|_{p}^{2}\|e^{-it\partial_{x}^{2}}v_{n}-e^{-it\partial_{x}^{2}}w_{n}\|_{p},$$
and by taking the $l^{q}$ and applying H\"older in the discrete variable with the embedding $l^{q}\hookrightarrow l^{2q}, l^{4q}$ and (\ref{Sch}), we have the estimate
$$\|\{\|e^{-it\partial_{x}^{2}}v_{n}\|_{p}\}_{n\in\Z}\|_{l^{4q}}^{2}\|\{\|e^{-it\partial_{x}^{2}}v_{n}-e^{-it\partial_{x}^{2}}w_{n}\|_{p}\}_{n\in\Z}\|_{l^{2q}}\leq$$
$$\|\{\|e^{-it\partial_{x}^{2}}v_{n}\|_{p}\}_{n\in\Z}\|_{l^{q}}^{2}\|\{\|e^{-it\partial_{x}^{2}}v_{n}-e^{-it\partial_{x}^{2}}w_{n}\|_{p}\}_{n\in\Z}\|_{l^{q}}=$$
$$\|e^{-it\partial_{x}^{2}}v_{n}\|_{M_{p,q}}^{2}\|e^{-it\partial_{x}^{2}}v_{n}-e^{-it\partial_{x}^{2}}w_{n}\|_{M_{p,q}}\lesssim(1+|t|)^{3|\frac12-\frac1{q}|}\|v_{n}\|_{M_{p,q}}^{2}\|v_{n}-w_{n}\|_{M_{p,q}}.$$
The operator difference $R_{2}^{t}(v)-R_{2}^{t}(w)$ is treated in a similar way and the proof is complete. 
\end{proof}
For the non-resonant part $N_{1}^{t}$ we split as
\begin{equation}
\label{main13}
N_{1}^{t}(v)(n)=N_{11}^{t}(v)(n)+N_{12}^{t}(v)(n),
\end{equation}
where 
$$N_{11}^{t}(v)(n)=\sum_{A_{N}(n)}Q^{1,t}_{n}(v_{n_{1}},\bar{v}_{n_{2}},v_{n_{3}}),$$
and 
\begin{equation}
\label{set1}
A_{N}(n)=\{(n_{1},n_{2},n_{3})\in\mathbb Z^3:n_{1}-n_{2}+n_{3}\approx n, n_{1}\not\approx n\not\approx n_{3}, |\Phi(n,n_{1},n_{2},n_{3})|\leq N\}.
\end{equation}
We will also denote by
\begin{equation}
\label{idid}
A_{N}(n)^{c}=\{(n_{1},n_{2},n_{3})\in\mathbb Z^3:n_{1}-n_{2}+n_{3}\approx n, n_{1}\not\approx n\not\approx n_{3}, |\Phi(n,n_{1},n_{2},n_{3})|> N\}.
\end{equation}
The number $N>0$ is considered to be large and will be fixed at the end of the proof. With the use of inequality (\ref{num}) we estimate $N_{11}^{t}$ as follows.

\begin{lemma}
\label{lemle}
$$\|N_{11}^{t}(v)\|_{l^{q}M_{p,q}}\lesssim (1+|t|)^{4|\frac12-\frac1{p}|}N^{\frac1{q'}+}\|v\|^{3}_{M_{p,q}},$$
and
$$\|N_{11}^{t}(v)-N_{11}^{t}(w)\|_{l^{q}M_{p,q}}\lesssim (1+|t|)^{4|\frac12-\frac1{p}|}N^{\frac1{q'}+}(\|v\|^{2}_{M_{p,q}}+\|w\|^{2}_{M_{p,q}})\|v-w\|_{M_{p,q}}.$$
\end{lemma}
\begin{proof}
Since $\|N_{11}^{t}(v)\|_{M_{p,q}}\leq\sum_{A_{N}(n)}\|Q_{n}^{t}(v_{n_{1}},\bar{v}_{n_{2}},v_{n_{3}})\|_{M_{p,q}}$ it suffices to estimate
$$\|Q_{n}^{t}(v_{n_{1}},\bar{v}_{n_{2}},v_{n_{3}})\|_{M_{p,q}}=\|e^{it\partial_{x}^{2}}\Box_{n}(u_{n_{1}}\bar{u}_{n_{2}}u_{n_{3}})\|_{M_{p,q}}\lesssim(1+|t|)^{|\frac12-\frac1{p}|}\|\Box_{n}(u_{n_{1}}\bar{u}_{n_{2}}u_{n_{3}})\|_{M_{p,q}}.$$
By writing out the last norm we have
$$\Big(\sum_{m\in\Z}\|\Box_{m}\Box_{n}(u_{n_{1}}\bar{u}_{n_{2}}u_{n_{3}})\|_{p}^{q}\Big)^{\frac1{q}}=\Big(\sum_{l\in\Lambda}\|\Box_{n+l}\Box_{n}(u_{n_{1}}\bar{u}_{n_{2}}u_{n_{3}})\|_{p}^{q}\Big)^{\frac1{q}}\lesssim$$
$$\|u_{n_{1}}\bar{u}_{n_{2}}u_{n_{3}}\|_{p}\leq\|u_{n_{1}}\|_{3p}\|u_{n_{2}}\|_{3p}\|u_{n_{3}}\|_{3p}\lesssim\|u_{n_{1}}\|_{p}\|u_{n_{2}}\|_{p}\|u_{n_{3}}\|_{p},$$
where we used Lemma (\ref{Bern}), H\"older, (\ref{Bern1}) and implication (\ref{Lamlam}). Therefore, we have
$$\sum_{A_{N}(n)}\|u_{n_{1}}\|_{p}\|u_{n_{2}}\|_{p}\|u_{n_{3}}\|_{p}\leq\Big(\sum_{A_{N}(n)}1^{q'}\Big)^{\frac1{q'}}\Big(\sum_{A_{N}(n)}\|u_{n_{1}}\|_{p}^{q}\|u_{n_{2}}\|_{p}^{q}\|u_{n_{3}}\|_{p}^{q}\Big)^{\frac1{q}}.$$
Fix $n$ and $\mu\in\mathbb Z$ such that $|\mu|\leq N$. From (\ref{num}) there are at most $o(N^{+})$ many choices for $n_{1}$ and $n_{3}$, and so for $n_{2}$ from $n\approx n_{1}-n_{2}+n_{3}$, satisfying 
$$\mu=2(n-n_{1})(n-n_{3}).$$
Thus, we arrive at
$$\|N_{11}^{t}(v)\|_{M_{p,q}}\lesssim(1+|t|)^{|\frac12-\frac1{p}|}N^{\frac1{q'}+}\Big(\sum_{A_{N}(n)}\|u_{n_{1}}\|_{p}^{q}\|u_{n_{2}}\|_{p}^{q}\|u_{n_{3}}\|_{p}^{q}\Big)^{\frac1{q}}$$
Then, we take the $l^{q}$ norm in the discrete variable and apply H\"older's inequality to obtain
$$\|N_{11}^{t}(v)\|_{l^{q}M_{p,q}}\lesssim(1+|t|)^{|\frac12-\frac1{p}|}N^{\frac1{q'}+}\Big(\sum_{n\in\Z}\sum_{A_{N}(n)}\|u_{n_{1}}\|_{p}^{q}\|u_{n_{2}}\|_{p}^{q}\|u_{n_{3}}\|_{p}^{q}\Big)^{\frac1{q}}$$
and Young's inequality (summation over $A_{N}(n)$ corresponds to convolution) provides us with the bound ($\|u_{n}\|_{M_{p,q}}\lesssim(1+|t|)^{|\frac12-\frac1{p}|}\|v_{n}\|_{M_{p,q}}$)
$$\|N_{11}^{t}(v)\|_{l^{q}M_{p,q}}\lesssim(1+|t|)^{4|\frac12-\frac1{p}|}N^{\frac1{q'}+}\|v\|_{M_{p,q}}^{3},$$
which finishes the proof.
\end{proof}
In order to continue, we look at the $N_{12}^{t}$ part more closely keeping in mind that we are on $A_{N}(n)^{c}.$ Our goal is to find a suitable splitting in order to continue our iteration. In the following we perform all formal calculations assuming that $v$ is a sufficiently smooth solution. Later, in Section \ref{thth6} we justify these formal computations also for $v\in C([0,T], M_{p,q}^{s}(\R))$, with $p,q$ and $s$ as in Theorem \ref{mainyeah}. 

From (\ref{main4}) and (\ref{malmal}) we may write 
$$\mathcal F(Q^{1,t}_{n}(v_{n_{1}},\bar{v}_{n_{2}},v_{n_{3}}))(\xi)=\sigma_{n}(\xi)\int_{\mathbb R^2}e^{-2it(\xi-\xi_{1})(\xi-\xi_{3})}\hat{v}_{n_{1}}(\xi_{1})\hat{\bar{v}}_{n_{2}}(\xi-\xi_{1}-\xi_{3})\hat{v}_{n_{3}}(\xi_{3})\ d\xi_{1}d\xi_{3},$$
and by the product rule we can write the previous integral as the sum of the following expressions 
\begin{equation}
\label{ttr}
\partial_{t}\Big(\sigma_{n}(\xi)\int_{\mathbb R^2}\frac{e^{-2it(\xi-\xi_{1})(\xi-\xi_{3})}}{-2i(\xi-\xi_{1})(\xi-\xi_{3})}\ \hat{v}_{n_{1}}(\xi_{1})\hat{\bar{v}}_{n_{2}}(\xi-\xi_{1}-\xi_{3})\hat{v}_{n_{3}}(\xi_{3})\ d\xi_{1}d\xi_{3}\Big)
\end{equation}
$$-\sigma_{n}(\xi)\int_{\mathbb R^2}\frac{e^{-2it(\xi-\xi_{1})(\xi-\xi_{3})}}{-2i(\xi-\xi_{1})(\xi-\xi_{3})}\partial_{t}\Big(\hat{v}_{n_{1}}(\xi_{1})\hat{\bar{v}}_{n_{2}}(\xi-\xi_{1}-\xi_{3})\hat{v}_{n_{3}}(\xi_{3})\Big)\ d\xi_{1}d\xi_{3}.$$ 
Therefore, we have the splitting 
\begin{equation}
\label{main5}
\mathcal F(Q^{1,t}_{n})=\partial_{t}\mathcal F(\tilde{Q}^{1,t}_{n})-\mathcal F(T^{1,t}_{n})
\end{equation}
or equivalently
\begin{equation}
\label{main6}
Q^{1,t}_{n}(v_{n_{1}},\bar{v}_{n_{2}},v_{n_{3}})=\partial_{t}(\tilde{Q}^{1,t}_{n}(v_{n_{1}},\bar{v}_{n_{2}},v_{n_{3}}))-T^{1,t}_{n}(v_{n_{1}},\bar{v}_{n_{2}},v_{n_{3}}),
\end{equation}
which allows us to write 
\begin{equation}
\label{nex}
N_{12}^{t}(v)(n)=\partial_{t}(N_{21}^{t}(v)(n))+N_{22}^{t}(v)(n),
\end{equation}
where
\begin{equation}
\label{nex1}
N_{21}^{t}(v)(n)=\sum_{A_{N}(n)^{c}}\tilde{Q}^{1,t}_{n}(v_{n_{1}},\bar{v}_{n_{2}},v_{n_{3}}),
\end{equation}
and
\begin{equation}
\label{nex2}
N_{22}^{t}(v)(n)=\sum_{A_{N}(n)^{c}}T_{n}^{1,t}(v_{n_{1}},\bar{v}_{n_{2}},v_{n_{3}}).
\end{equation}
In order to study these operators we have 
$$\mathcal F(\tilde{Q}^{1,t}_{n}(v_{n_{1}},\bar{v}_{n_{2}},v_{n_{3}}))(\xi)=e^{-it\xi^{2}}\sigma_{n}(\xi)\int_{\mathbb R^2}\frac{\hat{u}_{n_{1}}(\xi_{1})\hat{\bar{u}}_{n_{2}}(\xi-\xi_{1}-\xi_{3})\hat{u}_{n_{3}}(\xi_{3})}{(\xi-\xi_{1})(\xi-\xi_{3})}\ d\xi_{1}d\xi_{3},$$
and define 
\begin{equation}
\label{rr}
\mathcal F(R^{1,t}_{n}(u_{n_{1}},\bar{u}_{n_{2}},u_{n_{3}}))(\xi)=\sigma_{n}(\xi)\int_{\mathbb R^2}\frac{\hat{u}_{n_{1}}(\xi_{1})\hat{\bar{u}}_{n_{2}}(\xi-\xi_{1}-\xi_{3})\hat{u}_{n_{3}}(\xi_{3})}{(\xi-\xi_{1})(\xi-\xi_{3})}\ d\xi_{1}d\xi_{3},
\end{equation}
which is the same as the operator
\begin{equation}
\label{rr1}
R^{1,t}_{n}(u_{n_{1}},\bar{u}_{n_{2}},u_{n_{3}})(x)=\int_{\mathbb R^3}e^{ix\xi}\sigma_{n}(\xi)\frac{\hat{u}_{n_{1}}(\xi_{1})\hat{\bar{u}}_{n_{2}}(\xi-\xi_{1}-\xi_{3})\hat{u}_{n_{3}}(\xi_{3})}{(\xi-\xi_{1})(\xi-\xi_{3})}\ d\xi_{1}d\xi_{3}d\xi.
\end{equation}
At this point we introduce a fattened version of the $\sigma$-functions in the following way: Consider a function $\tilde\sigma_{0}$ with the same properties as $\sigma_{0}$ such that $\tilde\sigma_{0}\equiv1$ on the support of $\sigma_{0}$, $\mbox{supp}\tilde\sigma_{0}\subset B(0,\frac{17}{16})$ and define the tranlations $\tilde\sigma_{k}=\tilde\sigma_{0}(\cdot-k)$, $k\in\Z$. 

With this notation, writing out the Fourier transforms of the functions inside the integral in (\ref{rr1}), it is not difficult to see that
\begin{equation}
\label{main7}
R^{1,t}_{n}(u_{n_{1}},\bar{u}_{n_{2}},u_{n_{3}})(x)=\int_{\mathbb R^3}K^{(1)}_{n}(x,x_{1},y,x_{3})u_{n_{1}}(x_{1})\bar{u}_{n_{2}}(y)u_{n_{3}}(x_{3})\ dx_{1}dydx_{3},
\end{equation}
where
\begin{eqnarray*}
& & K^{(1)}_{n}(x,x_{1},y,x_{3}) \\ & = & \int_{\mathbb R^3}e^{i\xi_{1}(x-x_{1})+i\eta(x-y)+i\xi_{3}(x-x_{3})}\
\frac{\sigma_{n}(\xi_{1}+\eta+\xi_{3})}{(\eta+\xi_{1})(\eta+\xi_{3})}\ \tilde\sigma_{n_{1}}(\xi_{1})\tilde\sigma_{n_{2}}(-\eta)\tilde\sigma_{n_{3}}(\xi_{3})\ d\xi_{1}d\eta d\xi_{3} \\ &=& \mathcal F^{-1}\tilde\rho^{(1)}_{n}(x-x_{1},x-y,x-x_{3})
\end{eqnarray*}
and 
$$\tilde\rho_{n}^{(1)}(\xi_{1},\eta,\xi_{3})=\frac{\sigma_{n}(\xi_{1}+\eta+\xi_{3})}{(\eta+\xi_{1})(\eta+\xi_{3})}\tilde\sigma_{n_{1}}(\xi_{1})\tilde\sigma_{n_{2}}(-\eta)\tilde\sigma_{n_{3}}(\xi_{3}),\ \rho_{n}^{(1)}(\xi_{1},\eta,\xi_{3})=\frac{\sigma_{n}(\xi_{1}+\eta+\xi_{3})}{(\eta+\xi_{1})(\eta+\xi_{3})}.$$
The important estimate that the operator $\tilde Q^{1,t}_{n}$ satisfies is described in

\begin{lemma}
\label{fir}
For $2\leq p\leq\infty$
\begin{equation}
\label{ffir}
\|R^{1,t}_{n}(v_{n_{1}},\bar{v}_{n_{2}},v_{n_{3}})\|_{p}\lesssim\frac{\|v_{n_{1}}\|_{p}\|v_{n_{2}}\|_{p}\|v_{n_{3}}\|_{p}}{|n-n_{1}||n-n_{3}|},
\end{equation}
where the implicit constant depends on $p$. 
\end{lemma}
\begin{proof}
First, let us consider the case $p=2.$ This repeats the argument of the $M_{2,q}$ case treated in \cite{NP}. By duality, let $g\in L^2$, $\|g\|_{2}\neq0$, and consider the pairing
\begin{equation}
\label{thg}
|\langle R^{1,t}_{n}(v_{n_{1}},\bar{v}_{n_{2}},v_{n_{3}}), g\rangle|=\Big|\int_{\R}\mathcal F(R^{1,t}_{n}(v_{n_{1}},\bar{v}_{n_{2}},v_{n_{3}}))(\xi)\mathcal F(g)(\xi)\ d\xi\Big|
\end{equation}
$$=\Big|\int_{\R^3}\hat{g}(\xi)\ \sigma_{n}(\xi)\ \frac{\hat{v}_{n_{1}}(\xi_{1})\hat{\bar{v}}_{n_{2}}(\xi-\xi_{1}-\xi_{3})\hat{v}_{n_{3}}(\xi_{3})}{(\xi-\xi_{1})(\xi-\xi_{3})}\ d\xi d\xi_{1} d\xi_{3}\Big|$$
$$=\Big|\int_{\R^3}\hat{g}(\xi_{1}+\eta+\xi_{3})\ \frac{\sigma_{n}(\xi_{1}+\eta+\xi_{3})}{(\eta+\xi_{1})(\eta+\xi_{3})}\ \hat{v}_{n_{1}}(\xi_{1})\hat{\bar{v}}_{n_{2}}(\eta)\hat{v}_{n_{3}}(\xi_{3})\ d\eta d\xi_{1} d\xi_{3}\Big|$$
$$=\Big|\int_{I_{n_{1}}}\int_{I_{n_{2}}}\int_{I_{n_{3}}}\hat{g}(\xi_{1}+\eta+\xi_{3})\ \rho^{(1)}_{n}(\xi_{1},\eta,\xi_{3})\ \hat{v}_{n_{1}}(\xi_{1})\hat{\bar{v}}_{n_{2}}(\eta)\hat{v}_{n_{3}}(\xi_{3})\ d\xi_{1} d\eta d\xi_{3}\Big|,$$
where these three intervals are the compact supports of the functions $\hat{v}_{n_{1}},\hat{\bar{v}}_{n_{2}},\hat{v}_{n_{3}}$ (see (\ref{ww})). By H\"older's inequality we obtain the upper bound
$$\|\rho^{(1)}_{n}\|_{\infty}\|v_{n_{1}}\|_{2}\|v_{n_{2}}\|_{2}\|v_{n_{3}}\|_{2}\Big(\int_{I_{n_{1}}}\int_{I_{n_{2}}}\int_{I_{n_{3}}}|\hat{g}(\xi_{1}+\eta+\xi_{3})|^{2}\ d\xi_{1} d\eta d\xi_{3}\Big)^{\frac12},$$
and the last triple integral is easily estimated by
$$\|\hat{g}\|_{2}\ (|I_{n_{2}}||I_{n_{3}}|)^{\frac12}=\|g\|_{2}\ (|I_{n_{2}}||I_{n_{3}}|)^{\frac12}.$$
Therefore, the following is true
$$\|R^{1,t}_{n}(v_{n_{1}},\bar{v}_{n_{2}},v_{n_{3}}))\|_{2}\lesssim\|\rho_{n}^{(1)}\|_{\infty}\|v_{n_{1}}\|_{2}\|v_{n_{2}}\|_{2}\|v_{n_{3}}\|_{2},$$
and since $\xi_{1}\approx n_{1}$, $\eta\approx-n_{2}$ and $\xi_{3}\approx n_{3}$ we obtain
$$\|\rho_{n}^{(1)}\|_{\infty}\lesssim\frac1{|n-n_{1}||n-n_{3}|},$$
which finishes the proof.

Next let us consider the case $p=\infty$. Obviously, 
$$\|R^{1,t}_{n}(v_{n_{1}},\bar{v}_{n_{2}},v_{n_{3}})\|_{\infty}=\sup_{x\in\R}\Big|\int_{\R^3}(\mathcal F^{-1}\tilde\rho^{(1)}_{n})(x-x_{1},x-y,x-x_{3})v_{n_{1}}(x_{1})\bar{v}_{n_{2}}(y)v_{n_{3}}(x_{3})dx_{1}dydx_{3}\Big|,$$
which is bounded by
\begin{eqnarray*}
& & \sup_{x\in\R}\int_{\R^3}|(\mathcal F^{-1}\tilde\rho^{(1)}_{n})(x-x_{1},x-y,x-x_{3})|dx_{1}dydx_{3}\|v_{n_{1}}\|_{\infty}\|v_{n_{2}}\|_{\infty}\|v_{n_{3}}\|_{\infty}\\ &=&
\|\mathcal F^{-1}\tilde\rho_{n}^{(1)}\|_{L^{1}(\R^3)}\|v_{n_{1}}\|_{\infty}\|v_{n_{2}}\|_{\infty}\|v_{n_{3}}\|_{\infty}.
\end{eqnarray*}
By the embedding $H^{s}(\R^3)\hookrightarrow \mathcal FL^{1}(\R^3)$, for $s>3/2$, and the fact that $|\mbox{supp}(\tilde\rho_{n}^{(1)})|\lesssim 1$, it is sufficient to have an $L^{\infty}$ bound on the derivatives of $\tilde\rho_{n}^{(1)}$ of order $0, 1$ and $2$. Trivially, 
$$|\tilde\rho_{n}^{(1)}(\xi_{1},\eta,\xi_{3})|\lesssim\frac{1}{|n-n_{1}||n-n_{3}|},$$
since $\xi_{1}\approx n_{1}, \eta\approx-n_{2}$ and $\xi_{3}\approx n_{3}$. Then for the first order derivatives we get 
$$|\partial_{\xi_{j}}\tilde\rho_{n}^{(1)}|\lesssim\frac1{|\eta+\xi_{j}|^{2}|\eta+\xi_{4-j}|}+\frac{\|\sigma_{n}'\|_{\infty}+\|\tilde\sigma_{n_{j}}'\|_{\infty}}{|\eta+\xi_{j}||\eta+\xi_{4-j}|}\lesssim\frac{1}{|n-n_{1}||n-n_{3}|},$$
for $j=1,3$, since $|n-n_{1}|\geq1$. For the remaining derivative we observe that
$$|\partial_{\eta}\tilde\rho_{n}^{(1)}|\lesssim\frac{\|\sigma_{n}'\|_{\infty}+\|\tilde\sigma_{n_{2}}'\|_{\infty}}{|\eta+\xi_{1}||\eta+\xi_{3}|}+\frac{|2\eta+\xi_{1}+\xi_{3}|}{|\eta+\xi_{1}|^{2}|\eta+\xi_{3}|^{2}}\lesssim\frac{1}{|\eta+\xi_{1}||\eta+\xi_{3}|}+\frac{|\eta+\xi_{1}|+|\eta+\xi_{3}|}{|\eta+\xi_{1}|^{2}|\eta+\xi_{3}|^{2}},$$
which is bounded by
$$\frac{c}{|n-n_{1}||n-n_{3}|},$$
since $|\eta+\xi_{j}|\geq1$, where $c>0$ is a constant. Similarly we check the $2$nd order derivatives of $\tilde\rho_{n}^{(1)}$. Thus,
$$\|R^{1,t}_{n}(v_{n_{1}},\bar{v}_{n_{2}},v_{n_{3}})\|_{\infty}\lesssim \frac{\|v_{n_{1}}\|_{\infty}\|v_{n_{2}}\|_{\infty}\|v_{n_{3}}\|_{\infty}}{|n-n_{1}||n-n_{3}|}.$$
By interpolating between $p=2$ and $p=\infty$, we arrive at estimate (\ref{ffir}) for $2\leq p\leq\infty$. 
\end{proof}

\begin{remark}
\label{expl}
Notice that Lemma \ref{fir} (this observation applies to Lemma \ref{indu} too) is true for any triple of functions $f,g,h$ that lie in $M_{p,q}(\R)$ and the only important property is that they are nicely localised on the Fourier side since we consider their box operators $\Box_{n_{1}}f, \Box_{n_{2}}g$ and $\Box_{n_{3}}h.$ This observation will play an important role in the proof of Lemma \ref{finafinafina} of Section \ref{thth6}.
\end{remark}

Here is the estimate for the $N_{21}^{t}$ operator:

\begin{lemma}
\label{fir1}
$$\|N_{21}^{t}(v)\|_{l^{q}M_{p,q}}\lesssim (1+|t|)^{4|\frac12-\frac1{p}|}N^{\frac1{q'}-1+}\|v\|^{3}_{M_{p,q}},$$
and
$$\|N_{21}^{t}(v)-N_{21}^{t}(w)\|_{l^{q}M_{p,q}}\lesssim (1+|t|)^{4|\frac12-\frac1{p}|}N^{\frac1{q'}-1+}(\|v\|_{M_{p,q}}^{2}+\|w\|_{M_{p,q}}^{2})\|v-w\|_{M_{p,q}}.$$
\end{lemma}
\begin{proof}
Starting with the $M_{p,q}$ norm we have the estimate
$$\|N_{21}^{t}(v)\|_{M_{p,q}}\leq\sum_{A_{N}(n)^{c}}\|\tilde Q_{n}^{1,t}(v_{n_{1}},\bar{v}_{n_{2}},v_{n_{3}})\|_{M_{p,q}},$$
and the inner norm is equal to
$$\|\tilde Q_{n}^{1,t}(v_{n_{1}},\bar{v}_{n_{2}},v_{n_{3}})\|_{M_{p,q}}=\Big(\sum_{m\in\Z}\|\Box_{m}\tilde Q_{n}^{1,t}\|_{p}^{q}\Big)^{\frac1{q}}=\Big(\sum_{m\in\Z}\|\Box_{m}e^{it\partial_{x}^{2}} R_{n}^{1,t}\|_{p}^{q}\Big)^{\frac1{q}}\lesssim$$
$$(1+|t|)^{|\frac12-\frac1{p}|}\Big(\sum_{m\in\Z}\|\Box_{m}R_{n}^{1,t}\|_{p}^{q}\Big)^{\frac1{q}}=(1+|t|)^{|\frac12-\frac1{p}|}\Big(\sum_{m\in\Z}\|\mathcal F^{-1}\sigma_{m}\mathcal FR_{n}^{1,t}\|_{p}^{q}\Big)^{\frac1{q}},$$
from (\ref{Sch}). Since the Fourier transform of the operator $R_{n}^{1,t}$ is supported where $\sigma_{n}$ is, the last sum is actually a finite sum, that is
$$\Big(\sum_{l\in\Lambda}\|\Box_{n+l}R_{n}^{1,t}(u_{n_{1}},\bar{u}_{n_{2}},u_{n_{3}})\|_{p}^{q}\Big)^{\frac1{q}}\lesssim\|R_{n}^{1,t}(u_{n_{1}},\bar{u}_{n_{2}},u_{n_{3}})\|_{p}\lesssim\frac{\|u_{n_{1}}\|_{p}\|u_{n_{2}}\|_{p}\|u_{n_{3}}\|_{p}}{|n-n_{1}||n-n_{3}|},$$
by Lemma \ref{fir}. Then we take the $l^{q}$ norm in the discrete variable $n$ to arrive at the bound
$$\|N_{21}^{t}(v)\|_{l^{q}M_{p,q}}\lesssim(1+|t|)^{|\frac12-\frac1{p}|}\sum_{A_{N}(n)^{c}}\frac{\|u_{n_{1}}\|_{p}\|u_{n_{2}}\|_{p}\|u_{n_{3}}\|_{p}}{|n-n_{1}||n-n_{3}|},$$
and by H\"older's inequality we are led to the upper bound
$$(1+|t|)^{|\frac12-\frac1{p}|}\Big(\sum_{A_{N}(n)^{c}}\frac1{(|n-n_{1}||n-n_{3}|)^{q'}}\Big)^{\frac1{q'}}\Big(\sum_{A_{N}(n)^{c}}\|u_{n_{1}}\|_{p}^{q}\|u_{n_{2}}\|_{p}^{q}\|u_{n_{3}}\|_{p}^{q}\Big)^{\frac1{q}}.$$
The first sum (for $\mu=|n-n_{1}||n-n_{3}|$) is estimated with the use of (\ref{num}) from above by
$$\Big(\sum_{\mu=N+1}^{\infty}\frac{\mu^{\epsilon}}{\mu^{q'}}\Big)^{\frac1{q'}}\sim (N^{\epsilon+1-q'})^{\frac1{q'}}=N^{\frac1{q'}-1+},$$
and then with the use of Young's inequality we arrive at
$$\|N_{21}^{t}(v)\|_{l^{q}M_{p,q}}\lesssim(1+|t|)^{|\frac12-\frac1{p}|}N^{\frac1{q'}-1+}\|u\|^{3}_{M_{p,q}}\lesssim(1+|t|)^{4|\frac12-\frac1{p}|}N^{\frac1{q'}-1+}\|v\|^{3}_{M_{p,q}},$$
where we used (\ref{Sch}) ($u_{n}=e^{-it\partial_{x}^{2}}v_{n}$) and the proof is complete.
\end{proof}
To the remaining part $N_{22}^{t}$ we have to make use of equality (\ref{mainmain}) depending on whether the derivative falls on $\hat{v}_{n_{1}}$ or $\hat{\bar{v}}_{n_{2}}$ or $\hat{v}_{n_{3}}$. Let us see how we can proceed from here:
$$N_{22}^{t}(v)(n)=-2i\sum_{A_{N}(n)^{c}}\Big[\tilde{Q}^{1,t}_{n}(R^{t}_{2}(v)(n_{1})-R^{t}_{1}(v)(n_{1}),\bar{v}_{n_{2}},v_{n_{3}})+\tilde{Q}^{1,t}_{n}(N_{1}^{t}(v)(n_{1}),\bar{v}_{n_{2}},v_{n_{3}})\Big]$$
plus the corresponding term for $\partial_{t}\hat{\bar{v}}_{n_{2}}$ (the number $2$ that appears in front of the previous sum is because the expression is symmetric with respect to $v_{n_{1}}$ and $v_{n_{3}}$). Therefore, we can write $N_{22}^{t}$ as a sum
\begin{equation}
\label{patel2}
N_{22}^{t}(v)(n)=N_{4}^{t}(v)(n)+N_{3}^{t}(v)(n),
\end{equation}
where $N_{4}^{t}(v)(n)$ is the sum with the resonant part $R^{t}_{2}-R^{t}_{1}$. The following lemma is true

\begin{lemma}
\label{fir2}
$$\|N_{4}^{t}(v)\|_{l^{q}M_{p,q}}\lesssim (1+|t|)^{7|\frac12-\frac1{p}|}N^{\frac1{q'}-1+}\|v\|_{M_{p,q}}^{5},$$
and
$$\|N_{4}^{t}(v)-N_{4}^{t}(w)\|_{l^{q}M_{p,q}}\lesssim (1+|t|)^{7|\frac12-\frac1{p}|}N^{\frac1{q'}-1+}(\|v\|_{M_{p,q}}^{4}+\|w\|_{M_{p,q}}^{4})\|v-w\|_{M_{p,q}}.$$
\end{lemma}
\begin{proof}
Follows by Lemmata \ref{lem} and \ref{fir1} in the sense that we repeat the proof of Lemma \ref{fir1} and apply Lemma \ref{lem} to the part $R_{2}^{t}(v)(n_{1})-R_{1}^{t}(v)(n_{1})$.
\end{proof}
To continue, we have to decompose $N_{3}^{t}$ even further. It consists of three sums depending on which function the operator $N_{1}^{t}$ acts. One of them is the following (similar considerations apply for the remaining sums too)
\begin{equation}
\label{newnew1}
\sum_{A_{N}(n)^{c}}\tilde{Q}^{1,t}_{n}(N_{1}^{t}(v)(n_{1}),\bar{v}_{n_{2}},v_{n_{3}}),
\end{equation}
where
$$N_{1}^{t}(v)(n_{1})=\sum_{m_{1}\not\approx n_{1}\not\approx m_{3}}Q^{1,t}_{n_{1}}(v_{m_{1}},\bar{v}_{m_{2}},v_{m_{3}}),$$
and $n_{1}\approx m_{1}-m_{2}+m_{3}$. Here we have to consider new restrictions on the frequencies $(m_{1},m_{2},m_{3},n_{2},n_{3})$ where the "new" triple of frequencies $m_{1},m_{2},m_{3}$ appears as a "child" of the frequency $n_{1}.$ Thus, for $\mu_{1}=\Phi(n,n_{1},n_{2},n_{3})$ and $\mu_{2}=\Phi(n_{1},m_{1},m_{2},m_{3})$ we define the set 
\begin{equation}
\label{setset1}
C_{1}=\{|\mu_{1}+\mu_{2}|\leq 5^{3}|\mu_{1}|^{1-\frac1{100}}\},
\end{equation}
and split the sum in (\ref{newnew1}) as 
\begin{equation}
\label{patel}
\sum_{A_{N}(n)^{c}}\sum_{C_{1}}\ldots+\sum_{A_{N}(n)^{c}}\sum_{C_{1}^{c}}\ldots=N_{31}^{t}(v)(n)+N_{32}^{t}(v)(n).
\end{equation}
The following holds

\begin{lemma}
\label{fir3}
$$\|N_{31}^{t}(v)\|_{l^{q}M_{p,q}}\lesssim (1+|t|)^{8|\frac12-\frac1{p}|}N^{\frac2{q'}-\frac1{100q'}-1+}\|v\|_{M_{p,q}}^{5},$$
and
$$\|N_{31}^{t}(v)-N_{31}^{t}(w)\|_{l^{q}M_{p,q}}\lesssim (1+|t|)^{8|\frac12-\frac1{p}|}N^{\frac2{q'}-\frac1{100q'}-1+}(\|v\|^{4}_{M_{p,q}}+\|w\|_{M_{p,q}}^{4})\|v-w\|_{M_{p,q}}.$$
\end{lemma}
\begin{proof}
From (\ref{num}) we know that for fixed $n$ and $\mu_{1}$, there are at most $o(|\mu_{1}|^{+})$ many choices for $n_{1}$ and $n_{3}$ and for fixed $n_{1}$ and $\mu_{2}$ there are at most $o(|\mu_{2}|^{+})$ many choices for $m_{1}$ and $m_{3}$. From (\ref{setset1}) we can control $\mu_{2}$ in terms of $\mu_{1}$, that is $|\mu_{2}|\sim|\mu_{1}|$. In addition, for fixed $|\mu_{1}|$ there are at most $O(|\mu_{1}|^{1-\frac1{100}})$ many choices for $\mu_{2}.$ Also,
$$\|N_{31}^{t}(v)\|_{M_{p,q}}\leq\sum_{A_{N}(n)^{c}}\sum_{C_{1}}\|\tilde{Q}^{1,t}_{n}(Q^{1,t}_{n_{1}}(v_{m_{1}},\bar{v}_{m_{2}},v_{m_{3}}),\bar{v}_{n_{2}},v_{n_{3}})\|_{M_{p,q}},$$
and by doing the same estimate as in the proof of Lemma (\ref{fir1}) for the norm 
$$\|\tilde{Q}^{1,t}_{n}(Q^{1,t}_{n_{1}}(v_{m_{1}},\bar{v}_{m_{2}},v_{m_{3}}),\bar{v}_{n_{2}},v_{n_{3}})\|_{M_{p,q}},$$
we arrive at the upper bound
$$\|N_{31}^{t}(v)\|_{M_{p,q}}\lesssim(1+|t|)^{|\frac12=\frac1{p}|}\sum_{A_{N}(n)^{c}}\sum_{C_{1}}\frac{\|e^{-it\partial_{x}^{2}}Q^{1,t}_{n_{1}}(v_{m_{1}},\bar{v}_{m_{2}},v_{m_{3}})\|_{p}\|u_{n_{2}}\|_{p}\|u_{n_{3}}\|_{p}}{|n-n_{1}||n-n_{3}|}$$
and the last sum is bounded above by
$$\Big(\sum_{\mu=N+1}^{\infty}\frac{\mu^{1-\frac1{100}+}}{\mu^{q'}}\Big)^{\frac1{q'}}\Big(\sum_{A_{N}(n)^{c}}\sum_{C_{1}}\|e^{-it\partial_{x}^{2}}Q^{1,t}_{n_{1}}(v_{m_{1}},\bar{v}_{m_{2}},v_{m_{3}})\|_{p}^{q}\|u_{n_{2}}\|_{2}^{q}\|u_{n_{3}}\|_{p}^{q}\Big)^{\frac1{q}}.$$
Now we take the $l^{q}$ norm and apply Young's inequality for the second expression to arrive at the estimate
$$\|N_{31}^{t}(v)\|_{l^{q}M_{p,q}}\lesssim(1+|t|)^{4|\frac12-\frac1{p}|}N^{\frac2{q'}-\frac1{100q'}-1+}\|Q^{1,t}_{n_{1}}(v_{m_{1}},\bar{v}_{m_{2}},v_{m_{3}})\|_{M_{p,q}}\|v\|_{M_{p,q}}^{2},$$
and we treat the norm $\|Q^{1,t}_{n_{1}}(v_{m_{1}},\bar{v}_{m_{2}},v_{m_{3}})\|_{M_{p,q}}$ similarly as in Lemma (\ref{lem}) for the operator $R_{1}^{t}$ which finishes the proof.
\end{proof}
For the $N_{32}^{t}$ part we have to do the differentiation by parts technique which will create the $2$nd generation operators. Our first $2$nd generation operator $Q^{2,t}_{n}$ consists of $3$ sums 
$$q^{2,t}_{1,n}=\sum_{A_{N}(n)^{c}}\sum_{ C_{1}^{c}}\tilde{Q}^{1,t}_{n}(N_{1}^{t}(v)(n_{1}),\bar{v}_{n_{2}},v_{n_{3}}),$$
$$q^{2,t}_{2,n}=\sum_{A_{N}(n)^{c}}\sum_{ C_{1}^{c}}\tilde{Q}^{1,t}_{n}(v_{n_{1}},\overline{N_{1}^{t}(v)}(n_{2}),v_{n_{3}}),$$
$$q^{2,t}_{3,n}=\sum_{A_{N}(n)^{c}}\sum_{ C_{1}^{c}}\tilde{Q}^{1,t}_{n}(v_{n_{1}},\bar{v}_{n_{2}},N_{1}^{t}(v)(n_{3})).$$
Let us have a look at the first sum $q^{2,t}_{1,n}$ (we treat the other two in a similar manner). Its Fourier transform is equal to 
$$\sum_{A_{N}(n)^{c}}\sum_{ C_{1}^{c}}\sigma_{n}(\xi)\int_{\mathbb R^2}\frac{e^{-2it(\xi-\xi_{1})(\xi-\xi_{3})}}{(\xi-\xi_{1})(\xi-\xi_{3})}\ \mathcal F(N_{1}^{t}(v)(n_{1}))(\xi_{1})\hat{\bar{v}}_{n_{2}}(\xi-\xi_{1}-\xi_{3})\hat{v}_{n_{3}}(\xi_{3})\ d\xi_{1}d\xi_{3},$$
where
$$\mathcal F(N_{1}^{t}(v)(n_{1}))(\xi_{1})$$
equals
$$\sum_{\substack {n_{1}\approx m_{1}-m_{2}+m_{3} \\ m_{1}\not\approx n_{1}\not\approx m_{3}}}\sigma_{n_{1}}(\xi_{1})\int_{\mathbb R^2}e^{-2it(\xi_{1}-\xi_{1}')(\xi_{1}-\xi_{3}')}\hat{v}_{m_{1}}(\xi_{1}')\hat{\bar{v}}_{m_{2}}(\xi_{1}-\xi_{1}'-\xi_{3}')\hat{v}_{m_{3}}(\xi_{3}')\ d\xi_{1}'d\xi_{3}'.$$
Putting everything together and applying differentiation by parts we can write the integrals inside the sums as
$$\partial_{t}\Big(\sigma_{n}(\xi)\int_{\mathbb R^4}\sigma_{n_{1}}(\xi_{1})\frac{e^{-it(\mu_{1}+\mu_{2})}}{\mu_{1}(\mu_{1}+\mu_{2})}\hat{v}_{m_{1}}(\xi_{1}')\hat{\bar{v}}_{m_{2}}(\xi_{1}-\xi_{1}'-\xi_{3}')\hat{v}_{m_{3}}(\xi_{3}')\hat{\bar{v}}_{n_{2}}(\xi-\xi_{1}-\xi_{3})\hat{v}_{n_{3}}(\xi_{3})d\xi_{1}'d\xi_{3}'d\xi_{1}d\xi_{3}\Big)$$
minus 
$$\sigma_{n}(\xi)\int_{\mathbb R^4}\sigma_{n_{1}}(\xi_{1})\frac{e^{-it(\mu_{1}+\mu_{2})}}{\mu_{1}(\mu_{1}+\mu_{2})}\partial_{t}\Big(\hat{v}_{m_{1}}(\xi_{1}')\hat{\bar{v}}_{m_{2}}(\xi_{1}-\xi_{1}'-\xi_{3}')\hat{v}_{m_{3}}(\xi_{3}')\hat{\bar{v}}_{n_{2}}(\xi-\xi_{1}-\xi_{3})\hat{v}_{n_{3}}(\xi_{3})\Big)d\xi_{1}'d\xi_{3}'d\xi_{1}d\xi_{3},$$
where $\mu_{1}=(\xi-\xi_{1})(\xi-\xi_{3})$ and $\mu_{2}=(\xi_{1}-\xi_{1}')(\xi_{1}-\xi_{3}').$ Equivalently,
\begin{equation}
\label{sms}
\mathcal F(q^{2,t}_{1,n})=\partial_{t}(\tilde q^{2,t}_{1,n})-\mathcal F(\tau^{2,t}_{1,n}).
\end{equation}
Thus, by doing the same at the remaining two sums of $Q^{2,t}_{n}$, namely $q^{2,t}_{2,n}, q^{2,t}_{3,n}$, we obtain the splitting 
\begin{equation}
\label{neq11}
\mathcal F(Q^{2,t}_{n})=\partial_{t}\mathcal F(\tilde Q^{2,t}_{n})-\mathcal F(T^{2,t}_{n}).
\end{equation}
These new operators $\tilde q^{2,t}_{i,n}$, $i=1,2,3$, act on the following "type" of sequences
$$\tilde q^{2,t}_{1,n}(v_{m_{1}},\bar{v}_{m_{2}},v_{m_{3}},\bar{v}_{n_{2}},v_{n_{3}}),$$
with $m_{1}-m_{2}+m_{3}\approx n_{1}$ and $n_{1}-n_{2}+n_{3}\approx n$,
$$\tilde q^{2,t}_{2,n}(v_{n_{1}},\bar{v}_{m_{1}},v_{m_{2}},\bar{v}_{m_{3}},v_{n_{3}}),$$
with $m_{1}-m_{2}+m_{3}\approx n_{2}$ and $n_{1}-n_{2}+n_{3}\approx n$, and
$$\tilde q^{2,t}_{3,n}(v_{n_{1}}\bar{v}_{n_{2}},v_{m_{1}},\bar{v}_{m_{2}},v_{m_{3}}),$$
with $m_{1}-m_{2}+m_{3}\approx n_{3}$ and $n_{1}-n_{2}+n_{3}\approx n$. 

Writing out the Fourier transforms of the functions inside the integral of $\mathcal F(\tilde q^{2,t}_{1,n})$ it is not hard to see that 
$$\mathcal F(\tilde q^{2,t}_{1,n}(v_{m_{1}},\bar{v}_{m_{2}},v_{m_{3}},\bar{v}_{n_{2}},v_{n_{3}}))(\xi)=e^{-it\xi^{2}} \mathcal F(R^{2,t}_{n,n_{1}}(u_{m_{1}},\bar{u}_{m_{2}},u_{m_{3}},\bar{u}_{n_{2}},u_{n_{3}}))(\xi),$$
where the operator 
\begin{equation}
\label{set2}
R^{2,t}_{n,n_{1}}(u_{m_{1}},\bar{u}_{m_{2}},u_{m_{3}},\bar{u}_{n_{2}},u_{n_{3}})(x)=
\end{equation}
$$\int_{\mathbb R^5}K^{(2)}_{n,n_{1}}(x,x_{1}',y',x_{3}',y,x_{3})u_{m_{1}}(x_{1}')\bar{u}_{m_{2}}(y')u_{m_{3}}(x_{3}')\bar{u}_{n_{2}}(y)u_{n_{3}}(x_{3})\ dx_{1}'dy'dx_{3}'dydx_{3}$$
and the Kernel $K^{(2)}_{n,n_{1}}$ is given by the formula
\begin{equation}
\label{set3}
K^{(2)}_{n,n_{1}}(x,x_{1}',y',x_{3}',y,x_{3})=
\end{equation}
$$\int_{\mathbb R^5}[e^{i\xi_{1}'(x-x_{1}')+i\eta'(x-y')+i\xi_{3}'(x-x_{3}')+i\eta(x-y)+i\xi_{3}(x-x_{3})}]$$
$$\frac{\sigma_{n}(\xi_{1}'+\eta'+\xi_{3}'+\eta+\xi_{3})\sigma_{n_{1}}(\xi_{1}'+\eta'+\xi_{3}')\tilde\sigma_{m_{1}}(\xi_{1}')\tilde\sigma_{m_{2}}(-\eta')\tilde\sigma_{m_{3}}(\xi_{3}')\tilde\sigma_{n_{2}}(-\eta)\tilde\sigma_{n_{3}}(\xi_{3})}{(\eta+\eta'+\xi_{1}'+\xi_{3}')(\eta+\xi_{3})[(\eta+\eta'+\xi_{1}'+\xi_{3}')(\eta+\xi_{3})+(\eta'+\xi_{1}')(\eta'+\xi_{3}')]}$$ 
$$d\xi_{1}'d\eta'd\xi_{3}'d\eta d\xi_{3}=$$
$$(\mathcal F^{-1}\tilde\rho^{(2)}_{n,n_{1}})(x-x_{1}',x-y',x-x_{3}',x-y,x-x_{3}),$$
and the function $\tilde\rho^{(2)}_{n,n_{1}}$ equals
$$\tilde\rho^{(2)}_{n,n_{1}}=\frac{\sigma_{n}(\xi_{1}'+\eta'+\xi_{3}'+\eta+\xi_{3})\sigma_{n_{1}}(\xi_{1}'+\eta'+\xi_{3}')\tilde\sigma_{m_{1}}(\xi_{1}')\tilde\sigma_{m_{2}}(-\eta')\tilde\sigma_{m_{3}}(\xi_{3}')\tilde\sigma_{n_{2}}(-\eta)\tilde\sigma_{n_{3}}(\xi_{3})}{(\eta+\eta'+\xi_{1}'+\xi_{3}')(\eta+\xi_{3})[(\eta+\eta'+\xi_{1}'+\xi_{3}')(\eta+\xi_{3})+(\eta'+\xi_{1}')(\eta'+\xi_{3}')]}.$$
We also define the function
$$\rho^{(2)}_{n,n_{1}}(\xi_{1}',\eta',\xi_{3}',\eta,\xi_{3})=\frac{\sigma_{n}(\xi_{1}'+\eta'+\xi_{3}'+\eta+\xi_{3})\sigma_{n_{1}}(\xi_{1}'+\eta'+\xi_{3}')}{(\eta+\eta'+\xi_{1}'+\xi_{3}')(\eta+\xi_{3})[(\eta+\eta'+\xi_{1}'+\xi_{3}')(\eta+\xi_{3})+(\eta'+\xi_{1}')(\eta'+\xi_{3}')]}.$$
By the same calculations we obtain also the operators $R^{2,t}_{n,n_{2}}$ and $R^{2,t}_{n,n_{3}}$. They can be treated similarly to $R^{2,t}_{n,n_{1}}$ and for this reason in order to proceed we state a lemma for the operator $R^{2,t}_{n,n_{1}}$ as the one we had for $R^{1,t}_{n}$ (see Lemma \ref{fir}).

\begin{lemma}
\label{fir34}
For $2\leq p\leq\infty$
\begin{equation}
\|R^{2,t}_{n,n_{1}}(v_{m_{1}},\bar{v}_{m_{2}},v_{m_{3}},\bar{v}_{n_{2}},v_{n_{3}})\|_{p}\lesssim\frac{\|v_{m_{1}}\|_{p}\|v_{m_{2}}\|_{p}\|v_{m_{3}}\|_{p}\|v_{n_{2}}\|_{p}\|v_{n_{3}}\|_{p}}{|n-n_{1}||n-n_{3}||(n-n_{1})(n-n_{3})+(n_{1}-m_{1})(n_{1}-m_{3})|}.
\end{equation}
\end{lemma}
\begin{proof}
As in Lemma \ref{fir} we use interpolation between $L^{2}$ and $L^{\infty}$, and the only difference is that for the $L^{\infty}$ estimate we use the embedding of $H^{s}(\R^5)\hookrightarrow \mathcal FL^{1}(\R^5)$, for $s>5/2$, which means we have to calculate up to the $3$rd order derivative of the function $\tilde\rho^{(2)}_{n,n_{1}}$ in contrast to the function $\tilde\rho^{(1)}_{n}$ of Lemma \ref{fir} where we had to find all derivatives up to order $2$. 
\end{proof}
 
\begin{remark}
The operator $\tilde q^{2,t}_{3,n}$ satisfies exactly the same bound as $\tilde q^{2,t}_{1,n}$ since the only difference between these operators is a permutation of their variables. On the other hand, the operator $\tilde q^{2,t}_{2,n}$ is a bit different, since instead of taking only the permutation we have to conjugate the $2$nd variable too. Thus, a similar argument as the one given in Lemma \ref{fir34} leads to the estimate
\begin{equation}
\label{fir354}
\| R^{2,t}_{n,n_{2}}(v_{n_{1}},\bar{v}_{m_{1}},v_{m_{2}},\bar{v}_{m_{3}},v_{n_{3}})\|_{p}\lesssim\frac{\|v_{n_{1}}\|_{p}\|v_{m_{1}}\|_{p}\|v_{m_{2}}\|_{p}\|v_{m_{3}}\|_{p}\|v_{n_{3}}\|_{p}}{|(n-n_{1})(n-n_{3})||(n-n_{1})(n-n_{3})-(n_{2}-m_{1})(n_{2}-m_{3})|}
\end{equation}
which is not exactly the same as the one we had for the operators $R^{2,t}_{n,n_{1}}, R^{2,t}_{n,n_{3}}$ since in the denominator instead of having $\mu_{1}+\mu_{2}$ we have $\mu_{1}-\mu_{2}$ ($\mu_{1}=(n-n_{1})(n-n_{3})$ and in the first case $\mu_{2}=(n_{1}-m_{1})(n_{1}-\mu_{3})$, $m_{1}, m_{3}$ being the "children" of $n_{1}$, whereas in the second case $\mu_{2}=(n_{2}-m_{1})(n_{2}-m_{3})$, $m_{1}, m_{3}$ being the "children" of $n_{2}$). It is readily checked that this change in the sign does not really affect the calculations that are to follow.
\end{remark}
This lemma allows us to move forward with our iteration process and show that the operators
\begin{equation}
\label{formm1}
N_{0}^{(3)}(v)(n):=\sum_{A_{N}(n)^{c}}\sum_{C_{1}^{c}}\tilde Q^{2,t}_{n}=\sum_{A_{N}(n)^{c}}\sum_{C_{1}^{c}}\sum_{i=1}^{3}\tilde q^{2,t}_{i,n}
\end{equation}
and
\begin{equation}
\label{formm2}
N^{(3)}_{r}(v)(n):=\sum_{A_{N}(n)^{c}}\sum_{C_{1}^{c}}\Big(\tilde q^{2,t}_{1,n}(R^{t}_{2}(v)(m_{1})-R^{t}_{1}(v)(m_{1}),\bar{v}_{m_{2}},v_{m_{3}},\bar{v}_{n_{2}},v_{n_{3}})+
\end{equation}
$$\tilde q^{2,t}_{1,n}(v_{m_{1}},\overline{R^{t}_{2}(v)(m_{2})-R^{t}_{1}(v)(m_{2})},v_{m_{3}},\bar{v}_{n_{2}},v_{n_{3}})+\ldots+\tilde q^{2,t}_{3,n}(v_{n_{1}}\bar{v}_{n_{2}},v_{m_{1}},\bar{v}_{m_{2}},R^{t}_{2}(v)(m_{3})-R^{t}_{1}(v)(m_{3}))\Big),$$
are bounded on $l^{q}M_{p,q}$. The operator $N_{r}^{(3)}$ appears when we substitute each of the derivatives in the operator $\sum_{i=1}^{3}\tau^{2,t}_{i,n}$ by the expression given in (\ref{mainmain}). Notice that the operator $N_{0}^{(3)}$ has three summands and the operator $N_{r}^{(3)}$ has $3\cdot 5=15$ summands. Here is the claim

\begin{lemma}
\label{gg}
$$\|N_{0}^{(3)}(v)\|_{l^{q}M_{p,q}}\lesssim (1+|t|)^{6|\frac12-\frac1{p}|}N^{-2+\frac1{100}+\frac2{q'}-\frac1{100q'}+}\|v\|_{M_{p,q}}^{5},$$
and 
$$\|N_{0}^{(3)}(v)-N_{0}^{(3)}(w)\|_{l^{q}M_{p,q}}\lesssim (1+|t|)^{6|\frac12-\frac1{p}|}N^{-2+\frac1{100}+\frac2{q'}-\frac1{100q'}+}(\|v\|_{M_{p,q}}^{4}+\|w\|_{M_{p,q}}^{4})\|v-w\|_{M_{p,q}}.$$

$$\|N_{r}^{(3)}(v)\|_{l^{q}M_{p,q}}\lesssim (1+|t|)^{9|\frac12-\frac1{p}|}N^{-2+\frac1{100}+\frac2{q'}-\frac1{100q'}+}\|v\|^{7}_{M_{p,q}},$$
and
$$\|N_{r}^{(3)}(v)-N_{r}^{(3)}(w)\|_{l^{q}M_{p,q}}\lesssim (1+|t|)^{9|\frac12-\frac1{p}|}N^{-2+\frac1{100}+\frac2{q'}-\frac1{100q'}+}(\|v\|_{M_{p,q}}^{6}+\|w\|_{M_{p,q}}^{6})\|v-w\|_{M_{p,q}}.$$
\end{lemma}
\begin{proof}
Let us start with the operator $N_{0}^{(3)}$ and for simplicity of the presentation we will consider only the sum with the term $\tilde q^{2,t}_{1,n}.$ As in the proof of Lemma \ref{fir3} we have from (\ref{num}) that for fixed $n$ and $\mu_{1}$ there are at most $o(|\mu_{1}|^{+})$ many choices for $n_{1}, n_{2}, n_{3}$ (such that $(n-n_{1})(n-n_{3})=\mu_{1}$) and for fixed $n_{1}$ and $\mu_{2}$ there are at most $o(|\mu_{2}|^{+})$ many choices for $m_{1}, m _{2}, m_{3}$ (such that $(n_{1}-m_{1})(n_{1}-m_{3})=\mu_{2}$). Since the Fourier transform of the operator $\tilde q^{2,t}_{1,n}$ is localised around the interval $Q_{n}$, using the same argument as in Lemma \ref{fir1} together with Lemma \ref{fir34} we see that 
$$\sum_{A_{N}(n)^{c}}\sum_{C_{1}^{c}}\|\tilde q^{2,t}_{1,n}(v_{m_{1}},\bar{v}_{m_{2}},v_{m_{3}},\bar{v}_{n_{2}},v_{n_{3}})\|_{M_{p,q}}\lesssim$$
$$(1+|t|)^{|\frac12-\frac1{p}|}\sum_{A_{N}(n)^{c}}\sum_{C_{1}^{c}}\frac{\|u_{m_{1}}\|_{p}\|u_{m_{2}}\|_{p}\|u_{m_{3}}\|_{p}\|u_{n_{2}}\|_{p}\|u_{n_{3}}\|_{p}}{|n-n_{1}||n-n_{3}||(n-n_{1})(n-n_{3})+(n_{1}-m_{1})(n_{1}-m_{3})|}$$
and the sum of RHS is equal to 
$$\sum_{A_{N}(n)^{c}}\sum_{C_{1}^{c}}\frac{\|u_{m_{1}}\|_{p}\|u_{m_{2}}\|_{p}\|u_{m_{3}}\|_{p}\|u_{n_{2}}\|_{p}\|u_{n_{3}}\|_{p}}{|\mu_{1}||\mu_{1}+\mu_{2}|}$$
which by H\"older's inequality is bounded above by
$$\Big(\sum_{A_{N}(n)^{c}}\sum_{C_{1}^{c}}\frac1{|\mu_{1}|^{q'}|\mu_{1}+\mu_{2}|^{q'}}|\mu_{1}|^{+}|\mu_{2}|^{+}\Big)^{\frac1{q'}}\Big(\sum_{A_{N}(n)^{c}}\sum_{C_{1}^{c}}\|u_{m_{1}}\|_{2}^{q}\|u_{m_{2}}\|_{p}^{q}\|u_{m_{3}}\|_{p}^{q}\|u_{n_{2}}\|_{p}^{q}\|u_{n_{3}}\|_{p}^{q}\Big)^{\frac1{q}}.$$
By a very crude estimate it is not difficult to see that the first sum behaves like the number $N^{-2+\frac1{100}+\frac2{q'}-\frac1{100q'}+}$. Then, by taking the $l^{q}$ norm and applying Young's inequality for convolutions we are done. For the operator $N_{r}^{(3)}$ the proof is the same but in addition we use Lemma \ref{lem} for the operator $R_{2}^{t}-R_{1}^{t}$. 
\end{proof}
The operator that remains to be estimated is defined as

\begin{equation}
\label{formm3}
N^{(3)}(v)(n):=\sum_{A_{N}(n)^{c}}\sum_{C_{1}^{c}}\Big(\tilde q^{2,t}_{1,n}(N_{1}^{t}(v)(m_{1}),\bar{v}_{m_{2}},v_{m_{3}},\bar{v}_{n_{2}},v_{n_{3}})+
\end{equation}
$$\tilde q^{2,t}_{1,n}(v_{m_{1}},\overline{N_{1}^{t}(v)(m_{2})},v_{m_{3}},\bar{v}_{n_{2}},v_{n_{3}})+\ldots+\tilde q^{2,t}_{3,n}(v_{n_{1}}\bar{v}_{n_{2}},v_{m_{1}},\bar{v}_{m_{2}},N_{1}^{t}(v)(m_{3}))\Big),$$
which is the same as $N_{r}^{(3)}$ but in the place of the operator $R_{2}^{t}-R_{1}^{t}$ we have $N_{1}^{t}$. As before, we write
\begin{equation}
\label{formm4}
N^{(3)}=N_{1}^{(3)}+N_{2}^{(3)},
\end{equation}
where $N_{1}^{(3)}$ is the restriction of $N^{(3)}$ onto the set of frequencies
\begin{equation}
\label{setset2}
C_{2}=\{|\tilde\mu_{3}|\leq 7^{3}|\tilde\mu_{2}|^{1-\frac1{100}}\}\cup\{|\tilde\mu_{3}|\leq 7^{3}|\mu_{1}|^{1-\frac1{100}}\},
\end{equation}
where $\tilde\mu_{2}=\mu_{1}+\mu_{2}$ and $\tilde\mu_{3}=\mu_{1}+\mu_{2}+\mu_{3}$. The following is true:

\begin{lemma}
\label{gg1}
$$\|N_{1}^{(3)}(v)\|_{l^{q}M_{p,q}}\lesssim (1+|t|)^{10|\frac12-\frac1{p}|}N^{-2+\frac1{100}+\frac3{q'}-\frac2{100q'}+}\|v\|_{M_{p,q}}^{7},$$
and
$$\|N_{1}^{(3)}(v)-N_{1}^{(3)}(w)\|_{l^{q}M_{p,q}}\lesssim (1+|t|)^{10|\frac12-\frac1{p}|}N^{-2+\frac1{100}+\frac3{q'}-\frac2{100q'}+}(\|v\|^{6}_{M_{p,q}}+\|w\|^{6}_{M_{p,q}})\|v-w\|_{M_{p,q}}.$$
\end{lemma}
\begin{proof}
Let us only consider the very first summand of the operator $N_{1}^{(3)}$, that is the operator $\tilde q_{1,n}^{2,t}$ with $N_{1}^{t}$ acting on its first variable, since for the other summands similar considerations apply. For the proof we use again the divisor counting argument. From (\ref{num}) it follows that for fixed $n$ and $\mu_{1}$ there are at most $o(|\mu_{1}|^{+})$ many choices for $n_{1}, n_{2}, n_{3}$ ($\mu_{1}=(n-n_{1})(n-n_{3})$, $n= n_{1}-n_{2}+n_{3}$). For fixed $n_{1}$ and $\mu_{2}$ there are at most $o(|\mu_{2}|^{+})$ many choices for $m_{1}, m_{2}, m_{3}$ ($\mu_{2}=(n_{1}-m_{1})(n_{1}-m_{3})$, $n_{1}= m_{1}-m_{2}+m_{3}$) and for fixed $m_{1}$ and $\mu_{3}$ there are at most $o(|\mu_{3}|^{+})$ many choices for $k_{1}, k_{2}, k_{3}$ ($\mu_{3}=(m_{1}-k_{1})(m_{1}-k_{3})$, $m_{1}= k_{1}-k_{2}+k_{3}$). 

First, let us assume that our frequencies satisfy $|\tilde\mu_{3}|\lesssim |\tilde\mu_{2}|^{1-\frac1{100}}$. Since, $\tilde\mu_{3}=\tilde\mu_{2}+\mu_{3}$ we have $|\mu_{3}|\sim|\tilde\mu_{2}|$. Moreover, for fixed $|\tilde\mu_{2}|$ (equivalently, for fixed $\mu_{1}, \mu_{2}$) there are at most $O(|\tilde\mu_{2}|^{1-\frac1{100}})$ many choices for $\tilde\mu_{3}$ and hence, for $\mu_{3}=\tilde\mu_{3}-\tilde\mu_{2}$. In addition, $|\mu_{2}|\lesssim\max(|\mu_{1}|, |\tilde\mu_{2}|)$ and we should recall that since we are on $C_{1}^{c}$ we have $|\tilde\mu_{2}|=|\mu_{1}+\mu_{2}|>5^{3}|\mu_{1}|^{1-\frac1{100}}>5^{3}N^{1-\frac1{100}}$. Then by the same localisation argument as in the proof of Lemma \ref{fir1} together with Lemma \ref{fir34} we estimate the expression
$$\sum_{A_{N}(n)^{c}}\sum_{C_{1}^{c}}\sum_{C_{2}}\|\tilde q^{2,t}_{1,n}(Q^{1,t}_{m_{1}}(v_{k_{1}},\bar{v}_{k_{2}},v_{k_{3}}),\bar{v}_{m_{2}},v_{m_{3}},\bar{v}_{n_{2}},v_{n_{3}})\|_{M_{p,q}}$$
by
$$(1+|t|)^{|\frac12-\frac1{p}|}\sum_{A_{N}(n)^{c}}\sum_{C_{1}^{c}}\sum_{C_{2}}\frac{\|e^{-it\partial_{x}^{2}}Q^{1,t}_{m_{1}}(v_{k_{1}},\bar{v}_{k_{2}},v_{k_{3}})\|_{p}\|u_{m_{2}}\|_{p}\|u_{m_{3}}\|_{p}\|u_{n_{2}}\|_{p}\|u_{n_{3}}\|_{p}}{|n-n_{1}||n-n_{3}||(n-n_{1})(n-n_{3})+(n_{1}-m_{1})(n_{1}-m_{3})|}=$$
$$(1+|t|)^{|\frac12-\frac1{p}|}\sum_{A_{N}(n)^{c}}\sum_{C_{1}^{c}}\sum_{C_{2}}\frac{\|e^{-it\partial_{x}^{2}}Q^{1,t}_{m_{1}}(v_{k_{1}},\bar{v}_{k_{2}},v_{k_{3}})\|_{p}\|u_{m_{2}}\|_{p}\|u_{m_{3}}\|_{p}\|u_{n_{2}}\|_{p}\|u_{n_{3}}\|_{p}}{|\mu_{1}||\tilde\mu_{2}|}$$
and by H\"older's inequality we see that the sum is bounded above by
\begin{equation}
\label{usus}
\Big(\sum_{\substack{|\mu_{1}|>N \\ |\tilde\mu_{2}|>5^{3}N^{1-\frac1{100}}}}\frac{|\mu_{1}|^{+}|\mu_{2}|^{+}|\mu_{3}|^{+}|\tilde\mu_{2}|^{1-\frac1{100}}}{|\mu_{1}|^{q'}|\tilde\mu_{2}|^{q'}}\Big)^{\frac1{q'}}\times
\end{equation}
$$\Big(\sum_{A_{N}(n)^{c}}\sum_{C_{1}^{c}}\sum_{C_{2}}\|e^{-it\partial_{x}^{2}}Q^{1,t}_{m_{1}}(v_{k_{1}},\bar{v}_{k_{2}},v_{k_{3}})\|_{p}^{q}\|u_{m_{2}}\|_{p}^{q}\|u_{m_{3}}\|_{p}^{q}\|u_{n_{2}}\|_{p}^{q}\|u_{n_{3}}\|_{p}^{q}\Big)^{\frac1{q}}.$$
The first sum is controlled by
\begin{equation}
\label{estst}
\Big(\sum_{\substack{|\mu_{1}|>N \\ |\tilde\mu_{2}|>5^{3}N^{1-\frac1{100}}}}\frac1{|\mu_{1}|^{q'-\epsilon}|\tilde\mu_{2}|^{q'-1+\frac1{100}-\epsilon}}\Big)^{\frac1{q'}}\lesssim \Big(N^{3(1-\frac1{100})-q'(2-\frac1{100})+\frac1{100^{2}}+}\Big)^{\frac1{q'}}
\end{equation}
and with the use of Young's inequality at the second sum together with an estimate on the norm $\|e^{-it\partial_{x}^{2}}Q^{1,t}_{m_{1}}(v_{k_{1}},\bar{v}_{k_{2}},v_{k_{3}})\|_{M_{p,q}}$ we are done. 

On the other hand, if $|\tilde\mu_{3}|\lesssim|\mu_{1}|^{1-\frac1{100}}$, then for fixed $\mu_{1}, \mu_{2}$ there are at most $O(|\mu_{1}|^{1-\frac1{100}})$ many choices for $\tilde\mu_{3}$ and hence for $\mu_{3}.$ After this observation, the calculations are exactly the same as before but the first sum of (\ref{usus}) becomes
\begin{equation}
\label{estst1}
\Big(\sum_{\substack{|\mu_{1}|>N \\ |\tilde\mu_{2}|>5^{3}N^{1-\frac1{100}}}}\frac1{|\mu_{1}|^{q'-1+\frac1{100}-\epsilon}|\tilde\mu_{2}|^{q'-\epsilon}}\Big)^{\frac1{q'}}\lesssim \Big(N^{3-\frac2{100}-q'(2-\frac1{100})+}\Big)^{\frac1{q'}}.
\end{equation}
Between the two exponents of $N$ in (\ref{estst}) and (\ref{estst1}) we see that (\ref{estst1}) is the dominating one and the proof is complete.
\end{proof}
To the remaining part, namely $N_{2}^{(3)}$, we have to apply the differentiation by parts technique again. Note that here we only look at frequencies such that
$$|\tilde\mu_{3}|=|\mu_{1}+\mu_{2}+\mu_{3}|>7^{3}|\mu_{1}|^{1-\frac1{100}}>7^{3}N^{1-\frac1{100}},$$
or equivalently, frequencies that are on the set $C_{2}^{c}$. Instead, we will present the general $J$th step of the iteration procedure and prove the required Lemmata. To do this, we need to use the tree notation as it was introduced in \cite{GKO}.
\end{section}

\begin{section}{tree notation and the induction step}
\label{treeind}
A tree $T$ is a finite, partially ordered set with the following properties:

\begin{itemize}
\item For any $a_{1}, a_{2}, a_{3}, a_{4}\in T$ if $a_{4}\leq a_{2}\leq a_{1}$ and $a_{4}\leq a_{3}\leq a_{1}$ then $a_{2}\leq a_{3}$ or $a_{3}\leq a_{2}$. 
\item There exists a maximum element $r\in T$, that is $a\leq r$ for all $a\in T$ which is called the root. 
\end{itemize}
We call the elements of $T$ the \textbf{nodes} of the tree and in this content we will say that $b\in T$ is a \textbf{child} of $a\in T$ (or equivalently, that $a$ is the \textbf{parent} of $b$) if $b\leq a, b\neq a$ and for all $c\in T$ such that $b\leq c\leq a$ we have either $b=c$ or $c=a$. 

A node $a\in T$ is called \textbf{terminal} if it has no children. A \textbf{nonterminal} node $a\in T$ is a node with exactly $3$ children $a_{1}$, the left child, $a_{2}$, the middle child, and $a_{3}$, the right child. We define the sets
\begin{equation}
\label{setsetset}
T^{0}=\{\mbox{all nonterminal nodes}\},
\end{equation}
and
\begin{equation}
\label{setsetset1}
T^{\infty}=\{\mbox{all terminal nodes}\}.
\end{equation}
Obviously, $T=T^{0}\cup T^{\infty}$, $T^{0}\cap T^{\infty}=\emptyset$ and if $|T^{0}|=j\in\Z_{+}$ we have $|T|=3j+1$ and $|T^{\infty}|=2j+1$. We denote the collection of trees with $j$ parental nodes by
\begin{equation}
\label{setsetset2}
T(j)=\{T \ \mbox{is a tree with}\ |T|=3j+1\}.
\end{equation}
Next, we say that a sequence of trees $\{T_{j}\}_{j=1}^{J}$ is a \textbf{chronicle of} $J$ \textbf{generations} if:
\begin{itemize}
\item $T_{j}\in T(j)$ for all $j=1, 2, \ldots, J$.
\item $T_{j+1}$ is obtained by changing one of the terminal nodes of $T_{j}$ into a nonterminal node with exactly $3$ children, for all $j=1, 2, \ldots, J-1$.
\end{itemize}
Let us also denote by $\mathcal I(J)$ the collection of trees of the $J$th generation. It is easily checked by an induction argument that
\begin{equation}
\label{setsetset3}
|\mathcal I(J)|=1\cdot 3\cdot 5\ldots(2J-1)=:(2J-1)!!.
\end{equation}
Given a chronicle $\{T_{j}\}_{j=1}^{J}$ of $J$ generations we refer to $T_{J}$ as an \textbf{ordered tree of the} $J$\textbf{th} \textbf{generation}. We should keep in mind that the notion of ordered trees comes with associated chronicles. It includes not only the shape of the tree but also how it "grew". 

Given an ordered tree $T$ we define an \textbf{index function} $n:T\to\Z$ such that
\begin{itemize}
\item $n_{a}\approx n_{a_{1}}-n_{a_{2}}+n_{a_{3}}$ for all $a\in T^{0}$, where $a_{1}, a_{2}, a_{3}$ are the children of $a$,
\item $n_{a}\not\approx n_{a_{1}}$ and $n_{a}\not\approx n_{a_{3}}$, for all $a\in T^{0}$,
\item $|\mu_{1}|:=2|n_{r}-n_{r_{1}}||n_{r}-n_{r_{3}}|>N$, where $r$ is the root of $T$,
\end{itemize}
and we denote the collection of all such index functions by $\mathcal R(T)$. 

For the sake of completeness, as it was done in \cite{GKO}, given an ordered tree $T$ with the chronicle $\{T_{j}\}_{j=1}^{J}$ and associated index functions $n\in\mathcal R(T)$, we need to keep track of the generations of frequencies. Fix an $n\in\mathcal R(T)$ and consider the very first tree $T_{1}$. Its nodes are the root $r$ and its children $r_{1}, r_{2}, r_{3}$. We define the first generation of frequencies by 
$$(n^{(1)},n_{1}^{(1)},n_{2}^{(1)},n_{3}^{(1)}):=(n_{r},n_{r_{1}},n_{r_{2}},n_{r_{3}}).$$
From the definition of the index function we have
$$n^{(1)}\approx n_{1}^{(1)}-n_{2}^{(1)}+n_{3}^{(1)},\ n_{1}^{(1)}\not\approx n^{(1)}\not\approx n_{3}^{(1)}.$$
The ordered tree $T_{2}$ of the second generation is obtained from $T_{1}$ by changing one of its terminal nodes $a=r_{k}\in T_{1}^{\infty}$ for some $k=1,2,3$ into a nonterminal node. Then, the second generation of frequencies is defined by
$$(n^{(2)},n_{1}^{(2)},n_{2}^{(2)},n_{3}^{(2)}):=(n_{a},n_{a_{1}},n_{a_{2}},n_{a_{3}}).$$
Thus, we have $n^{(2)}=n_{k}^{(1)}$ for some $k=1,2,3$ and from the definition of the index function we have
$$n^{(2)}\approx n_{1}^{(2)}-n_{2}^{(2)}+n_{3}^{(2)},\ n_{1}^{(2)}\not\approx n^{(2)}\not\approx n_{3}^{(2)}.$$
This should be compared with what happened in the calculations we presented before when passing from the first step of the iteration process into the second step. Every time we apply the differentiation by parts technique we introduce a new set of frequencies. 

After $j-1$ steps, the ordered tree $T_{j}$ of the $j$th generation is obtained from $T_{j-1}$ by changing one of its terminal nodes $a\in T_{j-1}^{\infty}$ into a nonterminal node. Then, the $j$th generation frequencies are defined as 
$$(n^{(j)},n_{1}^{(j)},n_{2}^{(j)},n_{3}^{(j)}):=(n_{a},n_{a_{1}},n_{a_{2}},n_{a_{3}}),$$
and we have $n^{(j)}=n_{k}^{(m)}(=n_{a})$ for some $m=1,2,\ldots,j-1$ and $k=1,2,3$, since this corresponds to the frequency of some terminal node in $T_{j-1}$. In addition, from the definition of the index function we have
$$n^{(j)}\approx n_{1}^{(j)}-n_{2}^{(j)}+n_{3}^{(j)},\ n_{1}^{(j)}\not\approx n^{(j)}\not\approx n_{3}^{(j)}.$$
Finally, we use $\mu_{j}$ to denote the corresponding phase factor introduced at the $j$th generation. That is,
\begin{equation}
\label{muuu}
\mu_{j}=2(n^{(j)}-n_{1}^{(j)})(n^{(j)}-n_{3}^{(j)}),
\end{equation}
and we also introduce the quantities
\begin{equation}
\label{qqq}
\tilde\mu_{J}=\sum_{j=1}^{J}\mu_{j},\ \hat{\mu}_{J}=\prod_{j=1}^{J}\tilde\mu_{j}.
\end{equation}
We should keep in mind that everytime we apply differentiation by parts and split the operators, we need to control the new frequencies that arise from this procedure. For this reason we need to define the sets (see (\ref{setset1}) and (\ref{setset2})):
\begin{equation}
\label{sesee}
C_{J}:=\{|\tilde\mu_{J+1}|\leq(2J+3)^{3}|\tilde\mu_{J}|^{1-\frac1{100}}\}\cup\{|\tilde\mu_{J+1}|\leq(2J+3)^{3}|\mu_{1}|^{1-\frac1{100}}\}.
\end{equation}

Let us see how to use this notation and terminology in our calculations. On the very first step, $J=1$, we have only one tree, the root node $r$ and its three children $r_{1}, r_{2}, r_{3}$ (sometimes, when it is clear from the context, we will identify the nodes and the frequencies assigned to them, that is, we have the root $n=n_{r}$ and its three children $n_{r_{1}}=n_{1}, n_{r_{2}}=n_{2}, n_{r_{3}}=n_{3}$) and we have only one operator that needs to be controlled in order to proceed further, namely $\tilde q^{1,t}_{n}:=\tilde Q^{1,t}_{n}$. 

On the second step, $J=2$, we have three operators $\tilde q^{2,t}_{n,n_{1}}:=\tilde q^{2,t}_{1,n}, \tilde q^{2,t}_{n,n_{2}}:=\tilde q^{2,t}_{2,n}, \tilde q^{2,t}_{n,n_{3}}:=\tilde q^{2,t}_{3,n}$ that play the same role as $\tilde q^{1,t}_{n}$ did for the first step. Let us observe that for each one of these operators we must have estimates on their $L^{2}$ norms in order to be able and continue the iteration. These estimates were provided by Lemmata \ref{fir} and \ref{fir34}. 

On the general $J$th step we will have $|\mathcal I(J)|$ operators of the $\tilde q^{J,t}_{T^0,\mathbf n}$ "type" each one corresponding to one of the ordered trees of the $J$th generation, $T\in T(J)$, where $\mathbf n$ is an arbitrary fixed index function on $T.$ We have the subindices $T^0$ and $\mathbf n$ because each one of these operators has Fourier transform supported on the cubes with centers the frequencies assigned to the nodes that belong to $T^0$. 

Let us denote by $T_{\alpha}$ all the nodes of the ordered tree $T$ that are descendants of the node $\alpha\in T^{0}$, i.e. $T_{\alpha}=\{\beta\in T:\beta\leq\alpha,\ \beta\neq\alpha\}$. 

We also need to define the \textbf{principal and final "signs" of a node} $a\in T$ which are functions from the tree $T$ into the set $\{\pm1\}$:
\begin{equation}
\label{signsign}
\mbox{psgn}(a)=\begin{cases}
+1,\ a\ \mbox{is not the middle child of his father}\\
+1,\ a=r,\ \mbox{the root node}\\
-1,\ a\ \mbox{is the middle child of his father}
\end{cases}
\end{equation}
\begin{equation}
\label{signsignsign}
\mbox{fsgn}(a)=\begin{cases}
+1,\ \mbox{psgn}(a)=+1\ \mbox{and}\ a\ \mbox{has an even number of middle predecessors}\\
-1,\ \mbox{psgn}(a)=+1\ \mbox{and}\ a\ \mbox{has an odd number of middle predecessors}\\
-1,\ \mbox{psgn}(a)=-1\ \mbox{and}\ a\ \mbox{has an even number of middle predecessors}\\
+1,\ \mbox{psgn}(a)=-1\ \mbox{and}\ a\ \mbox{has an odd number of middle predecessors},
\end{cases}
\end{equation}
where the root node $r\in T$ is not considered a middle father. 
 
The operators $\tilde q^{J,t}_{T^0,\mathbf n}$ are defined through their Fourier transforms as
\begin{equation}
\label{oops}
\mathcal F(\tilde q^{J,t}_{T^0,\mathbf n}(\{w_{n_\beta}\}_{\beta\in T^{\infty}}))(\xi)=e^{-it\xi^{2}}\mathcal F(R^{J,t}_{T^0,\mathbf n}(\{e^{-it\partial_{x}^{2}}w_{n_\beta}\}_{\beta\in T^{\infty}}))(\xi),
\end{equation}
where the operator $R^{J,t}_{T^0,\mathbf n}$ acts on the functions $\{w_{n_\beta}\}_{\beta\in T^{\infty}}$ as
\begin{equation}
\label{oops1}
R^{J,t}_{T^0,\mathbf n}(\{w_{n_\beta}\}_{\beta\in T^{\infty}})(x)=\int_{\R^{2J+1}}K^{(J)}_{T^0}(x,\{x_{\beta}\}_{\beta\in T^{\infty}})\Big[\otimes_{\beta\in T^{\infty}}w_{n_\beta}(x_{\beta})\Big]\ \prod_{\beta\in T^{\infty}} dx_{\beta},
\end{equation}
and the kernel $K^{(J)}_{T^0,\mathbf n}$ is defined as 
\begin{equation}
\label{oopss}
K^{(J)}_{T^0,\mathbf n}(x,\{x_{\beta}\}_{\beta\in T^{\infty}})=\mathcal F^{-1}(\tilde\rho^{(J)}_{T^{0},\mathbf n})(\{x-x_{\beta}\}_{\beta\in T^{\infty}}).
\end{equation}
Here is the formula for the function $\tilde\rho^{(J)}_{T^0,\mathbf n}$ with ($|T^{\infty}|=2J+1$)-variables, $\xi_{\beta}$, $\beta\in T^{\infty}$:
\begin{equation}
\label{jc}
\tilde\rho^{(J)}_{T^0,\mathbf n}(\{\xi_{\beta}\}_{\beta\in T^{\infty}})=\Big[\prod_{\beta\in T^{\infty}}\tilde\sigma_{n_{\beta}}(\xi_{\beta})\Big]\Big[\prod_{\alpha\in T^0}\sigma_{n_{\alpha}}\Big(\sum_{\beta\in T^{\infty}\cap T_{\alpha}}\mbox{fsgn}(\beta)\ \xi_{\beta}\Big)\Big]\frac1{\hat{\mu}_{T}}.
\end{equation}
We also define the function
\begin{equation}
\label{jjcc}
\rho^{(J)}_{T^0,\mathbf n}(\{\xi_{\beta}\}_{\beta\in T^{\infty}})=\Big[\prod_{\alpha\in T^0}\sigma_{n_{\alpha}}\Big(\sum_{\beta\in T^{\infty}\cap T_{\alpha}}\mbox{fsgn}(\beta)\ \xi_{\beta}\Big)\Big]\frac1{\hat{\mu}_{T}},
\end{equation}
where we denote by 
\begin{equation}
\label{yeah}
\hat{\mu}_{T}=\prod_{\alpha\in T^0}\tilde\mu_{\alpha},\ \tilde\mu_{\alpha}=\sum_{\beta\in T^{0}\setminus T_{\alpha}}\mu_{\beta},
\end{equation}
and for $\beta\in T^{0}$ we have
\begin{equation}
\label{yyeah}
\mu_{\beta}=2(\xi_{\beta}-\xi_{\beta_{1}})(\xi_{\beta}-\xi_{\beta_{3}}),
\end{equation}
where we impose the relation $\xi_{\alpha}=\xi_{\alpha_{1}}-\xi_{\alpha_{2}}+\xi_{\alpha_{3}}$ for every $\alpha\in T^{0}$ that appears in the calculations until we reach the terminal nodes of $T^{\infty}.$ This is because in the definition of the function $ \rho^{J,t}_{T^0}$ we need the variables "$\xi$" to be assigned only at the terminal nodes of the tree $T.$ We use the notation $\mu_{\beta}$ in similarity to $\mu_{j}$ of equation (\ref{muuu}) because this is the "continuous" version of the discrete case. In addition, the variables $\xi_{\alpha_{1}}, \xi_{\alpha_{2}}, \xi_{\alpha_{3}}$ that appear in the expression (\ref{jc}) are supported in such a way that $\xi_{\alpha_{1}}\approx n_{\alpha_{1}}, \xi_{\alpha_{2}}\approx n_{\alpha_{2}}, \xi_{\alpha_{3}}\approx n_{\alpha_{3}}.$ This is because the functions $\sigma_{n_{\alpha}}$ are supported in such a way. Therefore, $|\hat{\mu}_{T}|\sim|\hat{\mu}_{J}|$. 

For the induction step of our iteration process we need the following lemma which should be compared with Lemmata \ref{fir} and \ref{fir34}.

\begin{lemma}
\label{indu}
For $2\leq p\leq\infty$
\begin{equation}
\| R^{J,t}_{T^0,\mathbf n}(\{v_{n_\beta}\}_{\beta\in T^{\infty}})\|_{p}\lesssim\Big(\prod_{\beta\in T^{\infty}}\|v_{n_\beta}\|_{p}\Big)\frac{\Big((J+1)!^{A}\ J^{\frac{3J}{2}}\Big)^{1-\frac2{p}}}{|\hat{\mu}_{T}|},
\end{equation}
for every tree $T\in T(J)$ and index function $\mathbf n\in\mathcal R(T)$.
\end{lemma}
\begin{proof}
We use interpolation between the $L^{2}$ estimate, which is done in exactly the same way as in Lemma \ref{fir}, and the $L^{\infty}$ estimate where we use that for $s>\frac{2J+1}{2}$ the embedding $H^{s}(\mathbb R^{2J+1})\hookrightarrow\mathcal FL^{1}(\mathbb R^{2J+1})$ is continuous. By H\"older's inequality the embedding constant is bounded above by the quantity
\begin{equation}
\label{spsps}
|\mathbb S^{2J}|^{\frac12}\ \Big(\int_{0}^{\infty}\frac{r^{2J}}{(1+r)^{2J+2}}\ dr\Big)^{\frac12},
\end{equation}
where $|\mathbb S^{2J}|$ denotes the surface measure of the $2J$-dimensional sphere in $\R^{2J+1}.$ It is known that
$$|\mathbb S^{2J}|=\frac{2^{J+1}\pi^{J}}{(2J-1)!!},$$
and the integral part of (\ref{spsps}) decays like a polynomial in $J$, which can be neglected compared to the double factorial decay of the surface measure of $\mathbb S^{2J}$. Thus, the embedding constant decays like $1/J^{\frac{J}{2}}$. 

Since the function $\tilde\rho^{(J)}_{T^0,\mathbf n}$ has $2J+1$ variables and consists of $4J+1$ factors and we have to calculate all possible derivatives of order $r$ up to the order $J+1$ we obtain 
$$\sum_{r=0}^{J+1}(2J+1)^{r}(4J+1)^{r}=\frac{[(2J+1)(4J+1)]^{J+2}-1}{(2J+1)(4J+1)-1}\sim J^{2J}$$
terms in total. Let us notice that the more distributed the derivatives are on the product of functions that consist the function $\tilde\rho^{(J)}_{T^0,\mathbf n}$ the smaller constants we obtain in terms of growth in $J$ compared to $(J+1)!^A.$ The factorial $(J+1)!^{A}$ appears in the calculations because we take $J+1$ derivatives of the $\sigma$-functions. Finally, let us observe that a factorial $(J+1)!$ appears in the calculations too, when all $J+1$ derivatives fall in terms of the form $1/x$, but since $A>1$, $(J+1)!^{A}$ dominates. 
\end{proof}
For the rest of the paper, let us use the notation
\begin{equation}
\label{newww}
d_{J}:=(J+1)!^{A}\ J^{\frac{3J}{2}}.
\end{equation}
By Stirling's formula we obtain that $d_{J}$ has the following behaviour for large $J$
\begin{equation}
\label{asymt}
d_{J}\sim(\sqrt{2\pi (J+1)})^{A}\ \Big(\frac{J+1}{e}\Big)^{A(J+1)}\ J^{\frac{3J}{2}}\sim\frac{J^{\frac{A}{2}}}{e^{AJ}}\ J^{(\frac32+A)J}.
\end{equation}

Given an index function $\mathbf n$ and $2J+1$ functions $\{v_{n_\beta}\}_{\beta\in T^{\infty}}$ and $\alpha\in T^{\infty}$ we define the action of the operator $N_{1}^{t}$ (see (\ref{main10})) on the set $\{v_{n_\beta}\}_{\beta\in T^{\infty}}$ to be the same set as before but with the difference that we have substituted the function $v_{n_\alpha}$ by the new function $N_{1}^{t}(v)(n_\alpha).$ We will denote this new set of functions $N_{1}^{t,\alpha}(\{v_{n_\beta}\}_{\beta\in T^{\infty}}).$ Similarly, the action of the operator $R_{2}^{t}-R_{1}^{t}$ (see (\ref{main9})) on the set of functions $\{v_{n_\beta}\}_{\beta\in T^{\infty}}$ will be denoted by $(R_{2}^{t,\alpha}-R_{1}^{t,\alpha})(\{v_{n_\beta}\}_{\beta\in T^{\infty}})$. 

The operator of the $J$th step, $J\geq 2$, that we want to estimate is given by the formula
\begin{equation}
\label{fina}
N_{2}^{(J)}(v)(n):=\sum_{T\in T(J-1)}\sum_{\alpha\in T^{\infty}}\sum_{\substack{\mathbf n\in\mathcal R(T)\\ \mathbf n_{r}=n}}\tilde q^{J-1,t}_{T^0}(N_{1}^{t,\alpha}(\{v_{n_\beta}\}_{\beta\in T^{\infty}})).
\end{equation}
Applying differentiation by parts on the Fourier side (keep in mind that from the splitting procedure we are on the sets $A_{N}(n)^{c},C_{1}^{c},\ldots,C_{J-1}^{c}$) we obtain the expression
\begin{equation}
\label{fina1}
N_{2}^{(J)}(v)(n)=\partial_{t}(N_{0}^{(J+1)}(v)(n))+N_{r}^{(J+1)}(v)(n)+N^{(J+1)}(v)(n), 
\end{equation}
where
\begin{equation}
\label{fina2}
N_{0}^{(J+1)}(v)(n):=\sum_{T\in T(J)}\sum_{\substack{\mathbf n\in\mathcal R(T)\\ \mathbf n_{r}=n}}\tilde q^{J,t}_{T^0,\mathbf n}(\{v_{n_\beta}\}_{\beta\in T^{\infty}}),
\end{equation}
and
\begin{equation}
\label{fina3}
N_{r}^{(J+1)}(v)(n):=\sum_{T\in T(J)}\sum_{\alpha\in T^{\infty}}\sum_{\substack{\mathbf n\in\mathcal R(T)\\ \mathbf n_{r}=n}}\tilde q^{J,t}_{T^0,\mathbf n}((R^{t,\alpha}_{2}-R^{t,\alpha}_{1})(\{v_{n_{\beta}}\}_{\beta\in T^{\infty}})),
\end{equation}
and
\begin{equation}
\label{fina4}
N^{(J+1)}(v)(n):=\sum_{T\in T(J)}\sum_{\alpha\in T^{\infty}}\sum_{\substack{\mathbf n\in\mathcal R(T)\\ \mathbf n_{r}=n}}\tilde q^{J,t}_{T^0,\mathbf n}(N_{1}^{t,\alpha}(\{v_{n_{\beta}}\}_{\beta\in T^{\infty}})).
\end{equation}
We also split the operator $N^{(J+1)}$ as the sum
\begin{equation}
\label{fina5}
N^{(J+1)}=N_{1}^{(J+1)}+N_{2}^{(J+1)},
\end{equation}
where $N_{1}^{(J+1)}$ is the restriction of $N^{(J+1)}$ onto $C_{J}$ and $N_{2}^{(J+1)}$ onto $C_{J}^{c}.$ First, we generalise Lemma \ref{gg} by estimating the operators $N_{0}^{(J+1)}$ and $N_{r}^{(J+1)}$

\begin{lemma}
\label{finaal}
$$\|N_{0}^{(J+1)}(v)\|_{l^{q}M_{p,q}}\lesssim \frac{d_{J}^{1-\frac2{p}}}{J^{(2-\frac3{q'})J}}(1+|t|)^{(2J+2)|\frac12-\frac1{p}|}N^{-\frac{(q'-1)}{q'}J+\frac{(q'-1)}{100q'}(J-1)+}\|v\|_{M_{p,q}}^{2J+1},$$
and
$$\|N_{0}^{(J+1)}(v)-N_{0}^{(J+1)}(w)\|_{l^{q}M_{p,q}}\lesssim$$ 
$$\frac{d_{J}^{1-\frac2{p}}}{J^{(2-\frac3{q'})J}}(1+|t|)^{(2J+2)|\frac12-\frac1{p}|} N^{-\frac{(q'-1)}{q'}J+\frac{(q'-1)}{100q'}(J-1)+}(\|v\|_{M_{p,q}}^{2J}+\|w\|_{M_{p,q}}^{2J})\|v-w\|_{M_{p,q}}.$$

$$\|N_{r}^{(J+1)}(v)\|_{l^{q}M_{p,q}}\lesssim \frac{d_{J}^{1-\frac2{p}}}{J^{(2-\frac3{q'})J}}(1+|t|)^{(2J+5)|\frac12-\frac1{p}|}N^{-\frac{(q'-1)}{q'}J+\frac{(q'-1)}{100q'}(J-1)+}\|v\|_{M_{p,q}}^{2J+3},$$
and
$$\|N_{r}^{(J+1)}(v)-N_{r}^{(J+1)}(w)\|_{l^{q}M_{p,q}}\lesssim$$ 
$$\frac{d_{J}^{1-\frac2{p}}}{J^{(2-\frac3{q'})J}}(1+|t|)^{(2J+5)|\frac12-\frac1{p}|} N^{-\frac{(q'-1)}{q'}J+\frac{(q'-1)}{100q'}(J-1)+}(\|v\|_{M_{p,q}}^{2J+2}+\|w\|_{M_{p,q}}^{2J+2})\|v-w\|_{M_{p,q}}.$$
\end{lemma}
\begin{proof}
As in the proof of Lemma \ref{gg} for fixed $n^{(j)}$ and $\mu_{j}$ there are at most $o(|\mu_{j}|^{+})$ many choices for $n_{1}^{(j)},n_{2}^{(j)},n_{3}^{(j)}.$ In addition, let us observe that $\mu_{j}$ is determined by $\tilde\mu_{1},\ldots,\tilde\mu_{j}$ and $|\mu_{j}|\lesssim\max(|\tilde\mu_{j-1}|,|\tilde\mu_{j}|)$, since $\mu_{j}=\tilde\mu_{j}-\tilde\mu_{j-1}.$ Then, for a fixed tree $T\in T(J)$, since the operator $\tilde q^{J,t}_{T^0,\mathbf n}$ has Fourier transform localised around the interval $Q_{n}$, using the same argument as in Lemma \ref{fir1} together with Lemma \ref{indu} we obtain the bound
(remember that $|\hat{\mu}_{T}|\sim|\hat{\mu}_{J}|=\prod_{k=1}^{J}|\tilde\mu_{k}|$)
$$\sum_{\substack{\mathbf n\in\mathcal R(T)\\ \mathbf n_{r}=n}}\|\tilde q^{J,t}_{T^0,\mathbf n}(\{v_{\beta}\}_{\beta\in T^\infty})\|_{M_{p,q}}\lesssim(1+|t|)^{|\frac12-\frac1{p}|}d_{J}^{1-\frac2{p}}\sum_{\substack{\mathbf n\in\mathcal R(T)\\ \mathbf n_{r}=n}}\Big(\prod_{\beta\in T^{\infty}}\|u_{n_\beta}\|_{p}\Big)\Big(\prod_{k=1}^{J}\frac1{|\tilde\mu_{k}|}\Big),$$
and by H\"older's inequality the sum is bounded from above by
\begin{equation}
\label{hahah}
\Big(\sum_{\substack{|\mu_{1}|>N\\ |\tilde\mu_{j}|>(2j+1)^{3}N^{1-\frac1{100}}\\ j=2,\ldots,J}}\prod_{k=1}^{J}\frac1{|\tilde\mu_{k}|^{q'}}|\mu_{k}|^{+}\Big)^{\frac1{q'}}\Big(\sum_{\substack{\mathbf n\in\mathcal R(T)\\ \mathbf n_{r}=n}}\prod_{\beta\in T^\infty}\|u_{n_{\beta}}\|_{p}^{q}\Big)^{\frac1{q}}.
\end{equation}
The first sum behaves like $N^{-\frac{(q'-1)}{q'}J+\frac{(q'-1)}{100q'}(J-1)+}$ and for the remaining part we take the $l^{q}$ norm in $n$ and by the use of Young's inequality we are done. 

At this point, let us observe the following: There is an extra factor $\sim J$ when we estimate the differences $N_{0}^{(J+1)}(v)-N_{0}^{(J+1)}(w)$ since $|a^{2J+1}-b^{2J+1}|\lesssim(\sum_{j=1}^{2J+1}a^{2J+1-j}b^{j-1})|a-b|$ has $O(J)$ many terms. Also, we have $c_{J}=|\mathcal I(J)|$ many summands in the operator $N_{0}^{(J+1)}$ since there are $c_{J}$ many trees of the $J$th generation and $c_{J}$ behaves like a double factorial, namely $(2J-1)!!$ (see (\ref{setsetset3})). However, these observations do not cause any problem since the constant that we obtain from estimating the first sum of (\ref{hahah}) decays like a fractional power of a double factorial in $J$, or to be more precise we have 
\begin{equation}
\label{factori}
\frac{c_{J}\ d_{J}^{1-\frac2{p}}}{\prod_{j=2}^{J}(2j+1)^{3\cdot\frac{q'-1}{q'}-}}=\frac{d_{J}^{1-\frac2{p}}}{(2J+1)^{3-\frac3{q'}-}[(2J-1)!!]^{2-\frac3{q'}}}\sim\frac{J^{(\frac32+A)(1-\frac2{p})J}}{J^{(2-\frac3{q'})J}}.
\end{equation}
In order to maintain the decay in the denominator we must have $2-\frac3{q'}-\frac32-A+\frac{2A+3}{p}>0$ which is equivalent to the restriction $p<\frac{2q'(2A+3)}{(2A-1)q'+6}$. This is true by the assumptions of Theorem \ref{th1} together with (\ref{mpap}). For the operator $N_{r}^{(J+1)}$ the proof is the same but in addition we use Lemma \ref{lem} for the operator $R_{2}^{t}-R_{1}^{t}$. 
\end{proof}
The estimate for the operator $N_{1}^{(J+1)}$, which generalises Lemma \ref{gg1}, is the following

\begin{lemma}
\label{finaal2}
$$\|N_{1}^{(J+1)}(v)\|_{l^{q}M_{p,q}}\lesssim \frac{d_{J}^{1-\frac2{p}}}{J^{(2-\frac3{q'})J}}(1+|t|)^{(2J+6)|\frac12-\frac1{p}|}N^{-1+\frac2{q'}-\frac1{100q'}+(1-\frac1{100})(\frac1{q'}-1)(J-1)+}\|v\|_{M_{p,q}}^{2J+3},$$
and

$$\|N_{1}^{(J+1)}(v)-N_{1}^{(J+1)}(w)\|_{l^{q}M_{p,q}}\lesssim \frac{d_{J}^{1-\frac2{p}}}{J^{(2-\frac3{q'})J}}(1+|t|)^{(2J+6)|\frac12-\frac1{p}|}$$
$$N^{-1+\frac2{q'}-\frac1{100q'}+(1-\frac1{100})(\frac1{q'}-1)(J-1)+}(\|v\|_{M_{p,q}}^{2J+2}+\|w\|_{M_{p,q}}^{2J+2})\|v-w\|_{M_{p,q}}.$$
\end{lemma}
\begin{proof}
As before, for fixed $n^{(j)}$ and $\mu_{j}$ there are at most $o(|\mu_{j}|^{+})$ many choices for $n_{1}^{(1)}, n_{2}^{(1)}, n_{3}^{(1)}$ and note that $\mu_{j}$ is determined by $\tilde\mu_{1},\ldots,\tilde\mu_{j}$. 

Let us assume that $|\tilde\mu_{J+1}|=|\tilde\mu_{J}+\mu_{J+1}|\lesssim(2J+3)^{3}|\tilde\mu_{J}|^{1-\frac1{100}}$ holds in (\ref{sesee}). Then, $|\mu_{J+1}|\lesssim|\tilde\mu_{J}|$ and for fixed $\tilde\mu_{J}$ there are at most $o(|\tilde\mu_{J}|^{1-\frac1{100}})$ many choices for $\tilde\mu_{J+1}$ and therefore, for $\mu_{J+1}=\tilde\mu_{J+1}-\tilde\mu_{J}.$ For a fixed tree $T\in T(J)$ and $\alpha\in T^{\infty}$, since the operator $\tilde q^{J,t}_{T^0,\mathbf n}$ has Fourier transform localised around the interval $Q_{n}$, by Lemma \ref{indu} we arrive at the upper bound (remember that $|\hat{\mu}_{T}|\sim|\hat{\mu}_{J}|=\prod_{k=1}^{J}|\tilde\mu_{k}|$)
$$\sum_{\substack{\mathbf n\in\mathcal R(T)\\ \mathbf n_{r}=n}}\|\tilde q^{J,t}_{T^0,\mathbf n}(N_{1}^{t,\alpha}(\{v_{n_{\beta}}\}_{\beta\in T^{\infty}}))\|_{M_{p,q}}\lesssim d_{J}^{1-\frac2{p}}(1+|t|)^{|\frac12-\frac1{p}|}$$
$$\sum_{\substack{\mathbf n\in\mathcal R(T)\\ \mathbf n_{r}=n}}\Big(\|e^{-it\partial_{x}^{2}}Q_{n_{\alpha}}^{1,t}(v_{n_{\alpha_{1}}}, \bar{v}_{n_{\alpha_{2}}}, v_{n_{\alpha_{3}}})\|_{p}\prod_{\beta\in T^{\infty}\setminus\{\alpha\}}\|u_{n_\beta}\|_{p}\Big)\Big(\prod_{k=1}^{J}\frac1{|\tilde\mu_{k}|}\Big),$$
and by H\"older's inequality we bound the sum by
\begin{equation}
\label{hahah2}
\Big(\sum_{\substack{|\mu_{1}|>N\\ |\tilde\mu_{j}|>(2j+1)^{3}N^{1-\frac1{100}}\\ j=2,\ldots,J}}|\tilde\mu_{J}|^{1-\frac1{100}+}\prod_{k=1}^{J}\frac1{|\tilde\mu_{k}|^{q'}}|\mu_{k}|^{+}\Big)^{\frac1{q'}}
\end{equation}
$$\sum_{\substack{\mathbf n\in\mathcal R(T)\\ \mathbf n_{r}=n}}\Big(\|e^{-it\partial_{x}^{2}}Q_{n_{\alpha}}^{1,t}(v_{n_{\alpha_{1}}}, \bar{v}_{n_{\alpha_{2}}}, v_{n_{\alpha_{3}}})\|_{p}^{q}\prod_{\beta\in T^{\infty}\setminus\{\alpha\}}\|u_{n_\beta}\|_{p}^{q}\Big)^{\frac1{q}}.$$
An easy calculation shows that the first sum behaves like $N^{-1+\frac2{q'}-\frac1{100q'}+(1-\frac1{100})(\frac1{q'}-1)(J-1)+}$ and then by taking the $l^q$ norm with the use of Young's inequality and an estimate on the norm $\|e^{-it\partial_{x}^{2}}Q_{n_{\alpha}}^{1,t}(v_{n_{\alpha_{1}}}, \bar{v}_{n_{\alpha_{2}}}, v_{n_{\alpha_{3}}})\|_{M_{p,q}}$ we are done. 

If $|\tilde\mu_{J+1}|\lesssim(2J+3)^{3}|\mu_{1}|^{1-\frac1{100}}$ holds in (\ref{sesee}), then for fixed $\mu_{j}$, $j=1,\ldots,J,$ there are at most $O(|\mu_{1}|^{1-\frac1{100}})$ many choices for $\mu_{J+1}.$ The same argument as above leads us to exactly the same expressions as in (\ref{hahah2}) but with the first sum replaced by the following
$$\Big(\sum_{\substack{|\mu_{1}|>N\\ |\tilde\mu_{j}|>(2j+1)^{3}N^{1-\frac1{100}}\\ j=2,\ldots,J}}|\mu_{1}|^{1-\frac1{100}}\prod_{k=1}^{J}\frac1{|\tilde\mu_{k}|^{q'}}|\mu_{k}|^{+}\Big)^{\frac1{q'}},$$
which again is bounded from above by $N^{-1+\frac2{q'}-\frac1{100q'}+(1-\frac1{100})(\frac1{q'}-1)(J-1)+}$ and the proof is complete.
\end{proof}

\begin{remark}
\label{reme}
For $s>0$ we have to observe that all previous Lemmata hold true if we replace the $l^{q}M_{p,q}$ norm by the $l^{q}_{s}M_{p,q}$ norm and the $M_{p,q}(\R)$ norm by the $M_{p,q}^{s}(\R)$ norm. To see this, consider $n^{(j)}$ large. Then, there exists at least one of $n_{1}^{(j)},n_{2}^{(j)},n_{3}^{(j)}$ such that $|n_{k}^{(j)}|\geq\frac13|n^{(j)}|$, $k\in\{1,2,3\}$, since we have the relation $n^{(j)}= n_{1}^{(j)}-n_{2}^{(j)}+n_{3}^{(j)}.$ Therefore, in the estimates of the $J$th generation, there exists at least one frequency $n_{k}^{(j)}$ for some $j\in\{1,\ldots,J\}$ with the property
$$\langle n\rangle^{s}\leq 3^{js}\langle n_{k}^{(j)}\rangle^{s}\leq 3^{Js}\langle n_{k}^{(j)}\rangle ^{s}.$$
This exponential growth does not affect our calculations due to the double factorial decay in the denominator of (\ref{factori}).
\end{remark}

\begin{remark}
Notice that all estimates that appear in the previous lemmata of this section are true for all values of $p\in[2,\infty]$, $q\in[1,\infty]$ and $s\geq 0$. 
\end{remark}
\end{section}

\begin{section}{existence of weak solutions in the extended sense}
\label{thth4}

In this subsection the calculations are the same as in \cite{GKO} (and \cite{NP}) where we just need to replace the $L^{2}$ (or the $M_{2,q}$) norm by the $M_{p,q}(\R)$ norm. We will present them for the sake of completion. 

Let us start by defining the partial sum operator $\Gamma_{v_{0}}^{(J)}$ as
\begin{equation}
\label{gamaa}
\Gamma_{v_{0}}^{(J)}v(t)=v_{0}+\sum_{j=2}^{J}N_{0}^{(j)}(v)(n)-\sum_{j=2}^{J}N_{0}^{(j)}(v_{0})(n)
\end{equation}
$$+\int_{0}^{t}R_{1}^{\tau}(v)(n)+R_{2}^{\tau}(v)(n)+\sum_{j=2}^{J}N_{r}^{(j)}(v)(n)+\sum_{j=1}^{J}N_{1}^{(j)}(v)(n)\ d\tau,$$
where we have $N_{1}^{(1)}:=N_{11}^{t}$ from (\ref{main13}), $N_{0}^{(2)}:=N_{21}^{t}$ from (\ref{nex}), $N_{1}^{(2)}:=N_{31}^{t}$ from (\ref{patel}) and $N_{r}^{(2)}:=N_{4}^{t}$ from (\ref{patel2}) and $v_{0}\in M_{p,q}(\R)$ is a fixed function. 

In the following we will denote by $X_{T}=C([0,T],M_{p,q}(\R))$. Our goal is to show that the series appearing on the RHS of (\ref{gamaa}) converge absolutely in $X_{T}$ for sufficiently small $T>0$, if $v\in X_{T}$, even for $J=\infty.$ Indeed, by Lemmata \ref{lem}, \ref{lemle}, \ref{finaal}, and \ref{finaal2} we obtain (we assume that $T<1$ so that the quantity $(1+T)^{(2J+6)|\frac12-\frac1{p}|}$ is an exponential in $J$ independent of $T$ which can be neglected by making $N$ possibly larger)
\begin{equation}
\label{argg}
\|\Gamma_{v_{0}}^{(J)}v\|_{X_{T}}\leq\|v_{0}\|_{M_{p,q}}+C\sum_{j=2}^{J}N^{-(1-\frac1{q'})(j-1)+\frac{q'-1}{100q'}(j-2)+}(\|v\|_{X_{T}}^{2j-1}+\|v_{0}\|_{M_{p,q}}^{2j-1})
\end{equation}
$$+CT\Big[\|v\|^{3}_{X_{T}}+\sum_{j=2}^{J}N^{-(1-\frac1{q'})(j-1)+\frac{q'-1}{100q'}(j-2)+}\|v\|_{X_{T}}^{2j+1}$$
$$+N^{\frac1{q'}+}\|v\|_{X_{T}}^{3}+\sum_{j=2}^{J}N^{-1+\frac2{q'}-\frac1{100q'}+(1-\frac1{100})(\frac1{q'}-1)(J-2)+}\|v\|_{X_{T}}^{2j+1}\Big].$$
Let us assume that $\|v_{0}\|_{M_{p,q}}\leq R$ and $\|v\|_{X_{T}}\leq\tilde R$, with $\tilde R\geq R\geq1$. From (\ref{argg}) we have
\begin{equation}
\label{argg2}
\|\Gamma_{v_{0}}^{(J)}v\|_{X_{T}}\leq R+CN^{\frac1{q'}-1+}R^{3}\sum_{j=0}^{J-2}(N^{\frac1{q'}-1+\frac{q'-1}{100q'}}R^{2})^{j}+CN^{\frac1{q'}-1+}\tilde R^{3}\sum_{j=0}^{J-2}(N^{\frac1{q'}-1+\frac{q'-1}{100q'}}\tilde R^{2})^{j}
\end{equation}
$$+CT\Big[(1+N^{\frac1{q'}+})\tilde R^{3}+CN^{\frac1{q'}-1+}\tilde R^{5}\sum_{j=0}^{J-2}(N^{\frac1{q'}-1+\frac{q'-1}{100q'}}\tilde R^{2})^{j}$$
$$+N^{\frac2{q'}-1-\frac1{100q'}+}\tilde R^{5}\sum_{j=0}^{J-2}(N^{\frac1{q'}-1+\frac{q'-1}{100q'}}\tilde R^{2})^{j}\Big].$$
We choose $N=N(\tilde R)$ large enough, such that $N^{\frac1{q'}-1+\frac{q'-1}{100q'}}\tilde R^{2}=N^{99\frac{1-q'}{100q'}}\tilde R^{2}\leq\frac12$, or equivalently,
\begin{equation}
\label{argg3}
N\geq(2\tilde R^{2})^{\frac{100q'}{99(q'-1)}},
\end{equation}
so that the geometric series on the RHS of (\ref{argg2}) converge and are bounded by $2.$ Therefore, we arrive at
\begin{equation}
\label{argg4}
\|\Gamma_{v_{0}}^{(J)}v\|_{X_{T}}\leq R+2CN^{\frac1{q'}-1+}R^{3}+2CN^{\frac1{q'}-1+}\tilde R^{3}
\end{equation}
$$+CT\Big[(1+N^{\frac1{q'}+})\tilde R^{2}+2N^{\frac1{q'}-1+}\tilde R^{4}+2N^{\frac{199-100q'}{100q'}+}\tilde R^{4}\Big]\tilde R,$$
and we choose $T>0$ sufficiently small such that
\begin{equation}
\label{argg4,5}
CT\Big[(1+N^{\frac1{q'}+})\tilde R^{2}+2N^{\frac1{q'}-1+}\tilde R^{4}+2N^{\frac{199-100q'}{100q'}+}\tilde R^{4}\Big]<\frac1{10}.
\end{equation}
With the use of (\ref{argg3}) we see that $2CN^{\frac1{q'}-1+}\tilde R^{3}\leq CN^{\frac{1-q'}{100q'}+}\tilde R$ and by further imposing $N$ to be sufficiently large such that
\begin{equation}
\label{argg5}
C N^{\frac{1-q'}{100q'}+}<\frac1{10},
\end{equation}
we have
\begin{equation}
\label{argg6}
\|\Gamma_{v_{0}}^{(J)}v\|_{X_{T}}\leq R+\frac{R}{10}+\frac{\tilde R}{5}=\frac{11}{10}R+\frac15 \tilde R.
\end{equation}
Thus, for sufficiently large $N$ and sufficiently small $T>0$ the partial sum operators $\Gamma_{v_{0}}^{(J)}$ are well defined in $X_{T}$, for every $J\in\mathbb N\cup\{\infty\}.$ We will write $\Gamma_{v_{0}}$ for $\Gamma_{v_{0}}^{(\infty)}$. 

Our next step is, given an initial datum $v_{0}\in M_{p,q}(\R)$ to construct a solution $v\in X_{T}$ in the sense of Definition \ref{def3}. To this end, let $s>\frac1{q'}$ (so that $M_{p,q}^{s}(\R)$ is a Banach algebra that embeds in $M_{p,q}(\R)\cap C_{b}(\R)$) and consider a sequence $\{v_{0}^{(m)}\}_{m\in\mathbb N}\in M_{p,q}^{s}(\R)\subset M_{p,q}(\R)$ whose Fourier transforms are all compactly supported (thus, all $v_{0}^{(m)}$ are smooth functions) and such that $v_{0}^{(m)}\to v_{0}$ in $M_{p,q}(\R)$ as $m\to\infty$. Let $R=\|v_{0}\|_{M_{p,q}}+1$ and we can assume that $\|v_{0}^{(m)}\|_{M_{p,q}}\leq R$, for all $m\in\mathbb N$. Denote by $v^{(m)}$ the local in time solution of NLS (\ref{maineq}) in $M_{p,q}^{s}(\R)$ with initial condition $v_{0}^{(m)}.$ It satisfies the Duhamel formula
\begin{equation}
\label{argg7}
v^{(m)}(t)=v_{0}^{(m)}+i\int_{0}^{t}N_{1}^{\tau}(v^{(m)})-R_{1}^{\tau}(v^{(m)})+R_{2}^{\tau}(v^{(m)})\ d\tau=
\end{equation}
$$v_{0}^{(m)}+\sum_{j=2}^{\infty}N_{0}^{(j)}(v^{(m)})(n)-\sum_{j=2}^{\infty}N_{0}^{(j)}(v_{0}^{(m)})(n)$$
$$+\int_{0}^{t}R_{1}^{\tau}(v^{(m)})(n)+R_{2}^{\tau}(v^{(m)})(n)+\sum_{j=2}^{\infty}N_{r}^{(j)}(v^{(m)})(n)+\sum_{j=1}^{\infty}N_{1}^{(j)}(v^{(m)})(n)\ d\tau=\Gamma_{v_{0}^{(m)}}v^{(m)}.$$
To see this it suffices to prove that the remainder term $N_{2}^{(J+1)}(v)$ given by (\ref{fina}) goes to zero in the $l^{q}M_{p,q}$ norm as $J$ goes to infinity for the smooth solutions $v^{(m)}.$ This will be done in Lemma \ref{finafinafina} of Section \ref{thth6} for rougher solutions too where it will be proved that the remainder term goes to zero for large $J$ in the $l^{\infty}M_{p,q}$ norm. 

Next we will show that this holds in $X_{T}$ for the same time $T=T(R)>0$ independent of $m\in\mathbb N$. Indeed, fix $m\in\mathbb N$ and observe that the norm $\|v^{(m)}\|_{X_{t}}=\|v^{(m)}\|_{C([0,t],M_{p,q})}$ is continuous in $t$. Since $\|v_{0}^{(m)}\|_{M_{p,q}}\leq R$ there is a time $T_{1}>0$ such that $\|v^{(m)}\|_{X_{T_{1}}}\leq 4R$. Then, by repeating the previous calculations with $\tilde R=4R$ and keeping one of the factors as $\|v^{(m)}\|_{X_{T_{1}}}$ we get
\begin{equation}
\label{argg8}
\|v^{(m)}\|_{X_{T_{1}}}=\|\Gamma_{v_{0}^{(m)}}v^{(m)}\|_{X_{T_{1}}}\leq\frac{11}{10}R+\frac15\|v^{(m)}\|_{X_{T_{1}}},
\end{equation}
if $N$ and $T_{1}$ satisfy (\ref{argg3}), (\ref{argg4,5}) and (\ref{argg5}). Therefore, we have
\begin{equation}
\label{argg9}
\|v^{(m)}\|_{X_{T_{1}}}\leq\frac{19}{10}R<2R.
\end{equation}
Thus, from the continuity of $t\to\|v^{(m)}\|_{X_{t}}$, there is $\epsilon>0$ such that $\|v^{(m)}\|X_{T_{1}+\epsilon}\leq 4R$. Then again, from (\ref{argg8}) and (\ref{argg9}) with $T_{1}+\epsilon$ in place of $T_{1}$ we derive that $\|v^{(m)}\|_{X_{T_{1}+\epsilon}}\leq 2R$ as long as $N$ and $T_{1}+\epsilon$ satisfy (\ref{argg3}), (\ref{argg4,5}) and (\ref{argg5}). By observing that these conditions are independent of $m\in\mathbb N$ we obtain a time interval $[0,T]$ such that $\|v^{(m)}\|_{X_{T}}\leq 2R$ for all $m\in\mathbb N$. 

A similar computation on the difference, by possibly taking larger $N$ and smaller $T$ leads to the estimate
\begin{equation}
\label{argg10}
\|v^{(m_{1})}-v^{(m_{2})}\|_{X_{T}}=\|\Gamma_{v_{0}^{(m_{1})}}v^{(m_{1})}-\Gamma_{v_{0}^{(m_{2})}}v^{(m_{2})}\|_{X_{T}}\leq
\end{equation}
$$(1+\frac1{10})\|v_{0}^{(m_{1})}-v_{0}^{(m_{2})}\|_{M_{p,q}}+\frac15\|v^{(m_{1})}-v^{(m_{2})}\|_{X_{T}},$$
which implies 
\begin{equation}
\label{argg11}
\|v^{(m_{1})}-v^{(m_{2})}\|_{X_{T}}\leq c\ \|v_{0}^{(m_{1})}-v_{0}^{(m_{2})}\|_{M_{p,q}},
\end{equation}
for some $c>0$ and therefore, the sequence $\{v^{(m)}\}_{m\in\mathbb N}$ is Cauchy in the Banach space $X_{T}$. Let us denote by $v^{\infty}$ its limit in $X_{T}$ and by $u^{\infty}=S(t)v^{\infty}$. We will show that $u^{\infty}$ satisfies NLS (\ref{maineq}) in the interval $[0,T]$ in the sense of Definition \ref{def3}. For convenience, we drop the superscript $\infty$ and write $u, v.$ In addition, let $u^{(m)}:=S(t)v^{(m)}$, where $v^{(m)}$ is the smooth solution to (\ref{mainmain}) with smooth initial data $v_{0}^{(m)}$ as described above and note that $u^{(m)}$ is the smooth solution to (\ref{maineq}) with smooth initial data $u_{0}^{(m)}:=v_{0}^{(m)}.$ Furthermore, $u^{(m)}\to u$ in $X_{T}$ because $v^{(m)}\to v$ in $X_{T}$ and since convergence in the modulation space $M_{p,q}(\R)$ implies convergence in the sense of distributions we conclude that $\partial_{x}u^{(m)}\to\partial_{x} u$ and $\partial_{t} u^{(m)}\to\partial_{t} u$ in $\mathcal S'((0,T)\times \R).$ Since $u^{(m)}$ satisfies NLS (\ref{maineq}) for every $m\in\mathbb N$ we have that
$$\mathcal N(u^{(m)})=u^{(m)}|u^{(m)}|^{2}=-i\partial_{t}u^{(m)}+\partial_{x}^{2}u^{(m)},$$
also converges to some distribution $w\in\mathcal S'((0,T)\times \R).$ Our claim is the following

\begin{proposition}
\label{argg12}
Let $w$ be the limit of $\mathcal N(u^{(m)})$ in the sense of distributions as $m\to\infty$. Then, $w=\mathcal N(u)$, where $\mathcal N(u)$ is to be interpreted in the sense of Definition \ref{def2}. 
\end{proposition}
\begin{proof}
Consider a sequence of Fourier cutoff multipliers $\{T_{N}\}_{N\in\mathbb N}$ as in Definition \ref{def1}. We will prove that 
$$\lim_{N\to\infty}\mathcal N(T_{N}u)=w,$$
in the sense of distributions. Let $\phi$ be a test function and $\epsilon>0$ a fixed given number. Our goal is to find $N_{0}\in\mathbb N$ such that for all $N\geq N_{0}$ we have
\begin{equation}
\label{argg13}
|\langle w-\mathcal N(T_{N}u), \phi\rangle|<\epsilon.
\end{equation}
The LHS can be estimated as
$$|\langle w-\mathcal N(T_{N}u), \phi\rangle|\leq|\langle w-\mathcal N(u^{(m)}), \phi \rangle|+|\langle \mathcal N(u^{(m)})-\mathcal N(T_{N}u^{(m)}), \phi\rangle |$$
$$+|\langle\mathcal N(T_{N}u^{m})-\mathcal N(T_{N}u), \phi\rangle |.$$
The first term is estimated very easily since by the definition of $w$ we have that
\begin{equation}
\label{argg14}
|\langle w-\mathcal N(u^{(m)}), \phi\rangle|<\frac13\ \epsilon,
\end{equation}
for sufficiently large $m\in\mathbb N$. 

To continue, let us consider the second summand for fixed $m$. By writing the difference $\mathcal N(u^{(m)})-\mathcal N(T_{N}u^{(m)})$ as a telescoping sum we have to estimate terms of the form
$$\Big|\int\int \Big[(I-T_{N})u^{(m)}\Big]\ |u^{(m)}|^{2}\ \phi\ dx\ dt\Big|,$$
where $I$ denotes the identity operator. This integral can be identified with the action of the distribution $\Big[(I-T_{N})u^{(m)}\Big]\ |u^{(m)}|^{2}\in M_{p,q}^{s}(\R)$ (which is a Banach algebra) onto the test function $\phi$, which in its turn can be controlled (H\"older's inequality) by the norms (up to constants)
$$\|\phi\|_{L^{2}_{T}M_{p',q'}}\|u^{(m)}\|_{L^{\infty}_{T}M_{p,q}^{s}}^{2}\|(I-T_{N})u^{(m)}\|_{L^{2}_{T}M_{p,q}^{s}}\lesssim $$
$$C_{\phi}\|u^{(m)}\|_{C((0,T),M_{p,q}^{s})}^{2}\|(I-T_{N})u^{(m)}\|_{L^{2}_{T}M_{p,q}^{s}}\lesssim C_{\phi, m}\|(I-T_{N})u^{(m)}\|_{L^{2}_{T}M_{p,q}^{s}}.$$
Here we have to observe that for every fixed $t$ the norm $\|(I-T_{N})u^{(m)}\|_{M_{p,q}^{s}}\to0$ as $N\to\infty$ and an application of Dominated Convergence Theorem in $L^{2}(0,T)$ implies that there is $N_{0}=N_{0}(m)$ with the property
\begin{equation}
\label{argg15}
C_{\phi, m}\|(I-T_{N})u^{(m)}\|_{L^{2}_{T}M_{p,q}^{s}}<\frac13\ \epsilon,
\end{equation}
for all $N\geq N_{0}$. 

For the last term, we need to observe two things. Firstly, let us consider the sequence $\{\mathcal N(T_{N}u^{(m)})\}_{m\in\mathbb N}$, for each fixed $N$. By applying the iteration process that we described in the previous subsection to $\{S(-t)\mathcal N(T_{N}u^{(m)})\}_{m\in\mathbb N}$, which is basically the nonlinearity in equation (\ref{mainmain}) up to the operator $T_{N}$, we see that $\{\mathcal N(T_{N}u^{(m)})\}_{m\in\mathbb N}$ is Cauchy in $\mathcal S'((0,T)\times\R)$, as $m\to\infty$ for each fixed $N\in\mathbb N$ since the sequence $u^{(m)}$ is Cauchy in $C((0,T),M_{p,q}(\R)).$ Since the operators $T_{N}$ are uniformly bounded in the $L^{p}$ norm in $N$ we conclude that this convergence is uniform in $N$. 

Secondly, let us observe that for fixed $N$, $T_{N}u$ is in $C((0,T), H^{\infty}(\R))$ since $u\in M_{p,q}(\R)$ and the multiplier $m_{N}$ of $T_{N}$ is compactly supported. Hence, $\mathcal N(T_{N}u)=T_{N}u|T_{N}u|^{2}$ makes sense as a function. Therefore, for fixed $N$ we obtain the upper bound
$$|\langle\mathcal N(T_{N}u^{(m)})-\mathcal N(T_{N}u), \phi\rangle|\leq$$
$$\|\phi\|_{L^{4}_{T}M_{p',q'}}(\|T_{N}u^{(m)}\|^{2}_{L^{4}_{T}M_{p,q}}+\|T_{N}u\|^{2}_{L^{4}_{T}M_{p,q}})\|T_{N}u^{m}-T_{N}u\|_{L^{4}_{T}M_{p,q}}\leq$$
$$C_{\phi, \|u\|_{X_{T}}}\|u^{(m)}-u\|_{C((0,T),M_{p,q})},$$
which can be made arbitrarily small. Hence, $\mathcal N(T_{N}u^{(m)})$ converges to $\mathcal N(T_{N}u)$ in $\mathcal S'((0,T)\times\R)$ as $m\to\infty$ for each fixed $N$. 

From these two observations we derive that $\mathcal N(T_{N}u^{(m)})\to \mathcal N(T_{N}u)$ in $\mathcal S'((0,T)\times\R)$ as $m\to\infty$ uniformly in $N.$ Equivalently, 
\begin{equation}
\label{argg16}
|\langle\mathcal N(T_{N}u^{(m)})-\mathcal N(T_{N}u), \phi\rangle|<\frac13\ \epsilon,
\end{equation}
for all large $m$, uniformly in $N$. Therefore, (\ref{argg13}) follows by choosing $m$ sufficiently large so that (\ref{argg14}) and (\ref{argg16}) hold, and then choosing $N_{0}=N_{0}(m)$ such that (\ref{argg15}) holds. 
\end{proof}
Finally, we have shown that the function $u=u^{\infty}$ is a solution to the NLS (\ref{maineq}) in the sense of Definition \ref{def3}. 
\end{section}

\begin{section}{unconditional uniqueness} 
\label{thth6}

In Sections \ref{firstep} and \ref{treeind} we switched the order of space integration with time differentiation and summation in the discrete variable with time differentiation too. In the following we justify these formal computations and obtain the unconditional wellposedness of Theorem \ref{mainyeah}. 

In this subsection we assume that $u_{0}\in M_{p,q}^{s}(\R)$ with either $s\geq0, 2\leq p\leq 3$ and $1\leq q\leq\frac32$ or $s>\frac23-\frac1{q}, 2\leq p\leq 3$ and $\frac32<q\leq\frac{18}{11}$ or $s>\frac23-\frac1{q}, 2\leq p<\frac{10q'}{q'+6}$ and $\frac{18}{11}<q\leq2$ which by (\ref{yeye233}) and (\ref{hhh}) implies that   
\begin{equation}
\label{jo}
M_{p,q}^{s}(\R)\hookrightarrow M_{3,\frac32}(\R)\hookrightarrow L^{3}(\R).
\end{equation}
By (\ref{jo}) we know that if $u$ is a solution of NLS (\ref{maineq}) in the space $C([0,T],M_{p,q}^{s}(\R))$ then $u$ and hence $v=e^{it\partial_{x}^{2}}u$ are elements of $X_{T}\hookrightarrow C([0,T], L^{3}(\R)).$ Thus, the nonlinearity of NLS (\ref{maineq}) makes sense as an element of $C([0,T], L^{1}(\R))$ and by (\ref{main3}) we obtain that $\partial_{t}v_{n}\in C([0,T],L^{1}(\R)).$ The next lemma justifies the interchange of time differentiation and space integration

\begin{lemma}
\label{didi}
Let $f(t,x),\partial_{t}f(t,x)\in C([0,T],L^{1}(\R^{d}))$ and define the distribution $\int_{\R^{d}}f(\cdot,x)dx$ by
$$\Big\langle \int_{\R^{d}}f(\cdot, x)dx, \phi\Big\rangle=\int_{\R}\int_{\R^{d}}f(t,x)\phi(t)dxdt,$$
with $\phi\in C^{\infty}_{c}(\R).$ Then, $\partial_{t}\int_{\R^{d}}f(\cdot,x)dx=\int_{\R^{d}}\partial_{t}f(\cdot,x)dx.$
\end{lemma}
\begin{proof}
By definition
$$\Big\langle\partial_{t}\int_{\R^{d}}f(\cdot,x)dx,\phi\Big\rangle=-\Big\langle\int_{\R^{d}}f(\cdot,x)dx,\phi'\Big\rangle=-\int_{\R}\int_{\R^{d}}f(t,x)\phi'(t)dxdt,$$
and since $f\in C([0,T],L^{1}(\R^{d}))$ we can change the order of integration by Fubini's Theorem and obtain
$$-\int_{\R^{d}}\int_{\R}f(t,x)\phi'(t)dtdx=\int_{\R^{d}}\int_{\R}\partial_{t}f(t,x)\phi(t)dtdx=\int_{\R}\int_{\R^{d}}\partial_{t}f(t,x)\phi(t)dxdt,$$
where in the first equality we used the definition of the weak derivative of $f$ and in the second equality Fubini's Theorem with the fact that $\partial_{t}f\in C([0,T],L^{1}(\R^{d})).$ The last integral is equal to 
$$\Big\langle\int_{\R^{d}}\partial_{t}f(\cdot,x)dx,\phi\Big\rangle,$$
and the proof is complete.
\end{proof}
Consider now (\ref{ttr}) for fixed $n$ and $\xi.$ We want to apply Lemma \ref{didi} to the function
$$f(t,\xi_{1},\xi_{3})=\sigma_{n}(\xi)\frac{e^{-2it(\xi-\xi_{1})(\xi-\xi_{3})}}{-2i(\xi-\xi_{1})(\xi-\xi_{3})}\ \hat{v}_{n_{1}}(\xi_{1})\hat{\bar{v}}_{n_{2}}(\xi-\xi_{1}-\xi_{3})\hat{v}_{n_{3}}(\xi_{3}),$$
where $\xi\approx n, \xi_{1}\approx n_{1},\xi_{3}\approx n_{3}, \xi-\xi_{1}-\xi_{3}\approx -n_{2}$ and $(n,n_{1},n_{2},n_{3})\in A_{N}(n)^{c}$ given by (\ref{idid}). Notice that $f, \partial_{t}f \in C([0,T],L^{1}(\R^{2}))$ since $v\in C([0,T],M_{2,q}^{s}(\R))$ and $\partial_{t}v_{n}\in C([0,T],L^{1}(\R))$ for all integers $n.$ Thus,
$$\partial_{t}\Big[\int_{\R^2}\sigma_{n}(\xi)\frac{e^{-2it(\xi-\xi_{1})(\xi-\xi_{3})}}{-2i(\xi-\xi_{1})(\xi-\xi_{3})}\ \hat{v}_{n_{1}}(\xi_{1})\hat{\bar{v}}_{n_{2}}(\xi-\xi_{1}-\xi_{3})\hat{v}_{n_{3}}(\xi_{3})d\xi_{1}d\xi_{3}\Big]=$$
$$\int_{\R^2}\sigma_{n}(\xi)\partial_{t}\Big[\sigma_{n}(\xi)\frac{e^{-2it(\xi-\xi_{1})(\xi-\xi_{3})}}{-2i(\xi-\xi_{1})(\xi-\xi_{3})}\ \hat{v}_{n_{1}}(\xi_{1})\hat{\bar{v}}_{n_{2}}(\xi-\xi_{1}-\xi_{3})\hat{v}_{n_{3}}(\xi_{3})\Big]d\xi_{1}d\xi_{3}=$$
$$\int_{\R^2}\sigma_{n}(\xi)\partial_{t}\Big[\frac{e^{-2it(\xi-\xi_{1})(\xi-\xi_{3})}}{-2i(\xi-\xi_{1})(\xi-\xi_{3})}\Big]\hat{v}_{n_{1}}(\xi_{1})\hat{\bar{v}}_{n_{2}}(\xi-\xi_{1}-\xi_{3})\hat{v}_{n_{3}}(\xi_{3})d\xi_{1}d\xi_{3}+$$
$$\int_{\R^2}\sigma_{n}(\xi)\frac{e^{-2it(\xi-\xi_{1})(\xi-\xi_{3})}}{-2i(\xi-\xi_{1})(\xi-\xi_{3})}\partial_{t}\Big[\hat{v}_{n_{1}}(\xi_{1})\hat{\bar{v}}_{n_{2}}(\xi-\xi_{1}-\xi_{3})\hat{v}_{n_{3}}(\xi_{3})\Big]d\xi_{1}d\xi_{3}.$$
In the second equality we used the product rule which is applicable since $v\in C([0,T],L^{3}(\R))$ implies that $\partial_{t}v_{n}\in C([0,T],L^{1}(\R)).$

Finally it remains to justify the interchange of differentiation in time and summation in the discrete variable but this is done in exactly the same way as in \cite{GKO} (Lemma $5.1$). Similar arguments justify the interchange on the $J$th step of the infinite iteration procedure. 

Thus, for $v\in C([0,T],M_{p,q}^{s}(\R))$ with $M_{p,q}^{s}(\R)\hookrightarrow L^{3}(\R)$ we can repeat the calculations of Sections \ref{firstep}, \ref{treeind} and \ref{thth4} to obtain the following expression in $X_{T}$ for the solution $u$ of NLS (\ref{maineq}) with initial data $u_{0}$
\begin{equation} 
\label{antt1}
u=\Gamma_{u_{0}}u+\lim_{J\to\infty}\int_{0}^{t}N_{2}^{(J+1)}(u)d\tau,
\end{equation}
where the limit is an element of $X_{T}.$ Its existence follows from the fact that the operators $\Gamma_{u_{0}}^{(J)}u$ converge to $\Gamma_{u_{0}}u$ in the norm of $X_{T}$ as $J\to\infty.$ The important estimate about the remainder operator $N_{2}^{(J)}$ is the following

\begin{lemma}
\label{finafinafina}
$$\lim_{J\to\infty}\|N_{2}^{(J)}(v)\|_{l^{\infty}M_{p,q}}=0.$$
\end{lemma} 
\begin{proof}
By (\ref{fina1}) we can write the remainder operator as the following sum
\begin{equation}
\label{ppp}
N_{2}^{(J)}(v)(n)=\partial_{t}(N_{0}^{(J+1)}(v)(n))+\sum_{T\in T(J)}\sum_{\alpha\in T^{\infty}}\sum_{\substack{\mathbf n\in\mathcal R(T)\\ \mathbf n_{r}=n}}\tilde q^{J,t}_{T^0,\mathbf n}(\partial_{t}^{(\alpha)}(\{v_{n_{\beta}}\}_{\beta\in T^{\infty}})),
\end{equation}
where we define the action of $\partial_{t}^{(\alpha)}$ onto the set of functions $\{v_{n_{\beta}}\}_{\beta\in T^{\infty}}$ to be the same set of functions except for the $\alpha$ node where we replace $v_{n_{\alpha}}$ by the function $\partial_{t}v_{n_{\alpha}}.$ 

We control the first summand $\partial_{t}(N_{0}^{(J+1)}(v)(n))$ by Lemma \ref{finaal}. For the last summand of the RHS of (\ref{ppp}) we estimate its $M_{p,q}(\R)$ norm in exactly the same way as in the proof of Lemma (\ref{finaal}) and arrive at the upper bound
$$(1+|t|)^{|\frac12-\frac1{p}|}d_{J}^{1-\frac{2}{p}}\sum_{T\in T(J)}\sum_{\alpha\in T^{\infty}}\sum_{\substack{\mathbf n\in\mathcal R(T)\\ \mathbf n_{r}=n}}\prod_{\beta\in T^{\infty}\setminus\{\alpha\}}\|u_{n_{\beta}}\|_{p}\ \frac{\|e^{-it\partial_{x}^{2}}\partial_{t}v_{n_{\alpha}}\|_{p}}{\prod_{k=1}^{J}|\tilde{\mu}_{k}|},$$
which by H\"older's inequality with exponents $\frac1{q}+\frac1{q'}=1$ implies
$$(1+|t|)^{|\frac12-\frac1{p}|}\ \frac{d_{J}^{1-\frac2{p}}}{J^{(3-\frac3{q'})J}}\sum_{T\in T(J)}\sum_{\alpha\in T^{\infty}}\Big(\sum_{\substack{\mathbf n\in\mathcal R(T)\\ \mathbf n_{r}=n}}\prod_{\beta\in T^{\infty}\setminus\{\alpha\}}\|u_{n_{\beta}}\|_{p}^{q}\|e^{-it\partial_{x}^{2}}\partial_{t}v_{n_{\alpha}}\|_{p}^{q}\Big)^{\frac1{q}}.$$
Then for the sum inside the parenthesis we apply Young's inequality in the discrete variable where for the first $2J$ functions we take the $l^{1}$ norm and for the last the $l^{\infty}$ norm we arrive at the estimate (from (\ref{main3}) we know that $\partial_{t}v_{n}=e^{it\partial_{x}^{2}}\Box_{n}(|u|^{2}u)$)
$$\|u\|_{M_{p,q}}^{2J}\sup_{n\in\Z}\|e^{-it\partial_{x}^{2}}\partial_{t}v_{n}\|_{p}=\|u\|_{M_{p,q}}^{2J}\|\Box_{n}(|u|^{2}u)\|_{l^{\infty}L^{p}}.$$
But $\|\Box_{n}(|u|^{2}u)\|_{p}\lesssim\|\Box_{n}(|u|^{2}u)\|_{1}\lesssim\||u|^{2}u\|_{1}=\|u\|_{3}^{3}\lesssim\|u\|_{M_{p,q}}^{3}$ where we used (\ref{Bern1}) and (\ref{Bern}). Using (\ref{Sch}) and putting everything together we arrive at the estimate
$$\|N_{2}^{(J)}(v)\|_{l^{\infty}M_{p,q}}\lesssim(1+|t|)^{(2J+3)|\frac12-\frac1{p}|}\ \frac{d_{J}^{1-\frac2{p}}}{J^{(2-\frac3{q'})J}}\ \|v\|_{M_{p,q}}^{2J+3},$$
which finishes the proof.
\end{proof}
This lemma implies that $\lim_{J\to\infty}\int_{0}^{t}N_{2}^{(J+1)}(u)d\tau$ is equal to $0$ in $X_{T}.$ From this we obtain the unconditional uniqueness of NLS (\ref{maineq}) since if there are two solutions $u_{1}$ and $u_{2}$ with the same initial datum $u_{0}$ we obtain by (\ref{argg11})
$$\|u_{1}-u_{2}\|_{X_{T}}=\|\Gamma_{u_{0}}u_{1}-\Gamma_{u_{0}}u_{2}\|_{X_{T}}\lesssim\|u_{0}-u_{0}\|_{M_{2,q}^{s}}=0.$$

\textbf{Acknowledgments}: The authors gratefully acknowledge financial support by the Deu\-tsche Forschungs\-gemeinschaft (DFG) through CRC 1173. 

\end{section}

\end{document}